\documentclass[11pt]{amsart}

\usepackage{amsfonts,amssymb,amsmath,amsthm}
\usepackage{bbm}
\usepackage{hhline}
\usepackage{stmaryrd}
\usepackage{longtable}
\usepackage{fancyhdr}

\input{xy}
\xyoption{all}

\begin{document}

\setlength{\parindent}{0pt} \setlength{\parskip}{1ex}

\newtheorem{Le}{Lemma}[section]
\newtheorem{Th}[Le]{Theorem}
\newtheorem{Thm}{Theorem}

\theoremstyle{definition}
\newtheorem{Def}[Le]{Definition}
\newtheorem{DaL}[Le]{Definition and Lemma}
\newtheorem{Conv}[Le]{Convention}
\newtheorem{Assumpt}[Le]{Assumption}

\theoremstyle{remark}
\newtheorem{Rem}[Le]{Remark}

\title[Exceptional holonomy and Aloff-Wallach spaces]{Exceptional
holonomy and Einstein metrics constructed from Aloff-Wallach
spaces}
\author{Frank Reidegeld}
\address{Fachbereich Mathematik, Bereich Analysis und Differentialgeometrie, Universit\"at Hamburg, Bundesstra\ss e 55, 20146
Hamburg, Germany} \email{frank.reidegeld\char"40
math.uni-hamburg.de} \subjclass[2000]{Primary 53C25, Secondary
53C29, 53C44. \\
\parbox{9pt}{$ $} This work was supported by the SFB 676 of the Deutsche
Forschungsgemeinschaft.\\}

\begin{abstract}
We investigate cohomogeneity-one metrics whose principal orbit is
an Aloff-Wallach space $SU(3)/U(1)$. In particular, we are
interested in metrics whose holonomy is contained in Spin($7$).
Complete metrics of this kind which are not product metrics have
exactly one singular orbit. We prove classification results for
metrics on tubular neighborhoods of various singular orbits. Since
the equation for the holonomy reduction has only few explicit
solutions, we make use of power series techniques. In order to
prove the convergence and the smoothness near the singular orbit,
we apply methods developed by Eschenburg and Wang. As a by-product
of these methods, we find many new examples of Einstein metrics of
cohomogeneity one.
\end{abstract}

\maketitle

\section{Introduction}

Metrics with holonomy Spin($7$) are an active area of research in
differential geometry and in superstring theory (see Acharya,
Gukov \cite{Ach}). Most of the explicitly known examples (see
Bazaikin \cite{Baz}, \cite{Baz1}; Bazaikin, Malkovich \cite{Baz2};
Bryant, Salamon \cite{BrS}; Cveti\v{c} et al. \cite{Cve02},
\cite{Cve}; Gukov, Sparks \cite{Guk}; Kanno, Yasui \cite{KanI},
\cite{KanII}) are of cohomogeneity one. The advantage of
cohomogeneity one metrics is that the equation for the holonomy
reduction is equivalent to a system of first-order ordinary
differential equations. Among those metrics, the metrics with an
Aloff-Wallach space as the principal orbit are an interesting
subclass.

An Aloff-Wallach space is a coset space
$N^{k,l}:=SU(3)/U(1)_{k,l}$ where $U(1)_{k,l}$ $:=
\{\text{diag}(e^{kit},e^{lit},e^{-i(k+l)t})|t\in \mathbb{R}\}$.
The spaces $N^{1,0}$ and $N^{1,1}$ have a different geometry as
the other ones and are called exceptional Aloff-Wallach spaces.
Any principal orbit of a manifold carrying a parallel
cohomogeneity-one Spin($7$)-structure, is equipped with a
cocalibrated homogeneous $G_2$-structure. We therefore will
describe how a connected component of the space of all
cocalibrated $SU(3)$-invariant $G_2$-structures on $N^{k,l}$ looks
like. This problem can be solved by means of representation
theory. After that we are able to deduce a system of ordinary
differential equations which is equivalent to the holonomy
reduction. For reasons of simplicity, we carry out this program
only for metrics which are diagonal with respect to a certain
basis.

If a cohomogeneity-one metric with holonomy Spin($7$) or a smaller
group is complete and not a product, it has exactly one singular
orbit. We therefore fix the initial values of our differential
equations at the singular orbit. Unfortunately, these initial
value problems have an explicit solution only in some special
cases \cite{Baz}, \cite{Baz2}, \cite{Cve}, \cite{Guk}.
Furthermore, the equations for the holonomy reduction degenerate
near the singular orbit. More precisely, some of the summands
behave like $\tfrac{0}{0}$. We therefore cannot apply the theorem
of Picard-Lindel\"of. In the literature \cite{Cve}, \cite{KanI},
\cite{KanII} there are indeed examples, where the solutions do not
only depend on the metric on the singular orbit but also on an
initial condition of second or third order which can be chosen
freely.

In order to solve these problems, we make a power series ansatz
for the metric. We have to check if the power series converges and
how many free parameters of higher order there are. Another
problem is that not any solution of the differential equations
corresponds to a metric which can be smoothly extended to the
singular orbit. There are certain smoothness conditions which have
to be satisfied and are in some cases a serious obstacle. If for
example the principal orbit is $N^{1,1}$ and the singular orbit is
$SU(3)/U(1)^2$, we have to replace the principal orbit by a
quotient $N^{1,1}/\mathbb{Z}_2$ in order to satisfy the smoothness
conditions.

The above problems were addressed by Eschenburg and Wang
\cite{Esch} in the context of cohomogeneity-one Einstein metrics.
Although metrics with exceptional holonomy are Ricci-flat and thus
Einstein metrics, we have to adapt the methods of \cite{Esch} to
our situation. After that we are finally able to prove the
existence of metrics whose holonomy is a subgroup of Spin($7$) on
a tubular neighborhood of the singular orbit. At this point we are
able to apply the main theorem of Eschenburg and Wang \cite{Esch}
and can also show the existence of Einstein metrics of
cohomogeneity one. With the exception of an $SU(3)$-invariant
Einstein metric on $\mathbb{HP}^2$ which can be found in
P\"uttmann, Rigas \cite{Puett}, all of the Einstein metrics are to
the best knowledge of the author new.

The main results of the article can be summarized as follows. We
find a two-parameter family of smooth non-homothetic
cohomogeneity-one metrics with holonomy a subgroup of Spin($7$).
All of them have $N^{1,1}/\mathbb{Z}_2$ as principal orbit and
$SU(3)/U(1)^2$ as singular orbit. Among them, there is a
one-parameter family of metrics with holonomy $SU(4)$. At the
border of the moduli space the metric converges to the Calabi
metric (see \cite{Calabi})on $T^{\ast} \mathbb{CP}^2$, which has
holonomy $Sp(2)$. First evidence for the above metrics can be
found in Kanno, Yasui \cite{KanII}. The two-parameter family and
the metrics with holonomy $SU(4)$ were investigated
independently of the author by Bazaikin and Malkovich
\cite{Baz}, \cite{Baz2}.

In \cite{Cve}, \cite{KanI} and \cite{KanII}, metrics with holonomy
Spin($7$) are constructed by numerical methods. One family of
metrics has principal orbit $N^{1,0}$ and singular orbit $S^5$.
Another family has an arbitrary $N^{k,l}$ with $(k,l)\neq(1,-1)$
as principal orbit and $\mathbb{CP}^2$ as singular orbit. These
are the metrics from \cite{Cve} and \cite{KanI} which depend on a
free parameter of third order and which we have already mentioned
above. The space on which the metrics are defined is an
$\mathbb{R}^4/\mathbb{Z}_{|k + l|}$-bundle over $\mathbb{CP}^2$
and thus an orbifold. If the principal orbit is the exceptional
Aloff-Wallach space $N^{1,1}$, there are further metrics with
singular orbit $\mathbb{CP}^2$ which are not of the above kind.
Examples of these metrics can be found in \cite{KanII}. There
exists a further two-parameter family of non-homothetic metrics
with that orbit structure which was found by Bazaikin \cite{Baz1}.
We also construct this metrics and show that the family can be
parameterized by two free parameters of third order.

Moreover, we prove for all of our cases with the help of the
methods of Eschenburg and Wang \cite{Esch} that there are no other
free parameters except the known ones, that the smoothness
conditions are satisfied and that the power series solutions
converge. We also prove that under certain conditions, for example
that the metric is diagonal, there are no further cohomogeneity
one metrics whose holonomy is contained in Spin($7$). Finally, we
prove that the holonomy of the metrics whose principal orbit is
not $N^{1,1}$ is all of Spin($7$).

The article is organized as follows. In the second section, we
collect some basic facts on $G_2$- and Spin($7$)-structures. In
Section \ref{CohomOneSection}, cohomogeneity-one manifolds and the
methods of Eschenburg and Wang \cite{Esch} are introduced. The
fourth section deals with the geometry of the Aloff-Wallach
spaces. Our metrics with cohomogeneity one are constructed in the
remaining three sections. In each section, we deal with metrics
which have one particular singular orbit. The fifth section is
about metrics with singular orbit $SU(3)/U(1)^2$, the sixth about
$S^5$, and the final section is about metrics with $\mathbb{CP}^2$
as singular orbit.

\textbf{Acknowledgements:} The article is based on the doctoral
thesis of the author \cite{ReiDiss}. I want to thank Lorenz
Schwachh\"ofer for providing me the subject of my thesis and his
many helpful advices as well as Yaroslav Bazaikin and Eugene
Malkovich for numerous interesting discussions.

\section{$G_2$- and $\text{Spin}(7)$-structures}

In this section, we define some terms which we will often use
later on. Let $(dx^1,\ldots,dx^7)$ be the standard basis of
one-forms on $\mathbb{R}^7$. Furthermore, let

\begin{equation}
\label{omega}
\omega := dx^{123}+dx^{145}-dx^{167}+dx^{246}+dx^{257}+dx^{347}-dx^{356}\:,
\end{equation}

where $dx^{i_1i_2\ldots i_k}$ denotes $dx^{i_1}\wedge\ldots\wedge dx^{i_k}$. The Hodge dual of $\omega$ is

\begin{equation}
\label{astomega} \ast\omega =
-dx^{1247}+dx^{1256}+dx^{1346}+dx^{1357}-dx^{2345}+dx^{2367}
+dx^{4567}\:.
\end{equation}

We supplement $(dx^1,\ldots,dx^7)$ with $dx^0$ to a basis of
one-forms on $\mathbb{R}^8$ and define

\begin{equation}
\label{Omega}
\begin{split}
\Omega & := \ast\omega+dx^0\wedge\omega\\
& = dx^{0123}+dx^{0145}-dx^{0167}+dx^{0246}+dx^{0257}+dx^{0347}-dx^{0356}\\
&\quad
-dx^{1247}+dx^{1256}+dx^{1346}+dx^{1357}-dx^{2345}+dx^{2367}
+dx^{4567}\:.\\
\end{split}
\end{equation}

A \emph{$G_2$-structure} on a seven-dimensional manifold is a
three-form which can be identified via local frames with $\omega$.
Analogously, a \emph{Spin($7$)-structure} on an eight-dimensional
manifold is a four-form which can be identified at each point with
$\Omega$. To any $G_2$-structure or Spin($7$)-structure we can
associate a canonical metric $g$ and an orientation. We call a
$G_2$-structure or Spin($7$)-structure $\omega$ or $\Omega$
\emph{parallel} if $\nabla^g\omega=0$ or $\nabla^g \Omega=0$. A
pair $(N,\omega)$ or $(M,\Omega)$ of a seven-dimensional or
eight-dimensional manifold and a parallel $G_2$-structure or
Spin($7$)-structure is called a \emph{$G_2$-manifold or
Spin($7$)-manifold}. Those manifolds have the following
interesting properties.

\begin{Th} (See Bonan \cite{Bon}; Fern\'andez, Gray \cite{Fer}; Fern\'andez \cite{Fer2})
\begin{enumerate}
    \item Let $N$ be a seven-dimensional manifold with a $G_2$-structure $\omega$. Then, the following statements on $\omega$
    are equivalent.
    \begin{enumerate}
        \item $\omega$ is parallel.
        \item $d\omega=d\ast\omega=0$.
        \item The holonomy of the associated metric $g$ is contained in $G_2$.
    \end{enumerate}
    Conversely, if $(N,g)$ is a Riemannian manifold with holonomy a subgroup of $G_2$, then there exists a parallel
    $G_2$-structure on $N$ such that its associated metric is $g$. If any of the above conditions is satisfied, $g$ is
    Ricci-flat.
    \item Let $M$ be an eight-dimensional manifold with a Spin($7$)-structure $\Omega$. Then, the following statements on
    $\Omega$ are equivalent.
    \begin{enumerate}
        \item $\Omega$ is parallel.
        \item $d\Omega=0$.
        \item The holonomy of the associated metric $g$ is contained in Spin($7$).
    \end{enumerate}
    Conversely, if $(M,g)$ is a Riemannian manifold with holonomy a subgroup of Spin($7$), then there exists a parallel
    Spin($7$)-structure on $N$ such that its associated metric is $g$. If any of the above conditions is satisfied, $g$ is
    Ricci-flat.
\end{enumerate}
\end{Th}

Finally, we introduce the following types of non-parallel $G_2$-structures which we will need later on.

\begin{Def}
A $G_2$-structure $\omega$ is called
\begin{enumerate}
    \item \emph{nearly parallel} if $d\omega=\lambda\ast\omega$ for a $\lambda\in\mathbb{R}\setminus\{0\}$.
    \item \emph{cocalibrated} if $d\ast\omega=0$.
\end{enumerate}
\end{Def}

\section{Cohomogeneity-one manifolds}
\label{CohomOneSection}

The set of all Spin($7$)-structures on a manifold $M$ does not define a vector subbundle of $\bigwedge^4 T^\ast M$.
Therefore, the condition $d\Omega=0$ should be considered as a non-linear partial differential equation. Explicit solutions of
this equation are hard to find. Among them are the first examples of complete metrics with holonomy Spin($7$) by Bryant
and Salamon \cite{BrS}. The existence of the non-explicit metrics of Bryant \cite{Br} and of Joyce \cite{Joy} can only be proven by
sophisticated analytic arguments. Our problem becomes a lot of simpler if we assume that $\Omega$ is of cohomogeneity one.

\begin{DaL} (Cf. Mostert \cite{Mostert} and references therein)
\begin{enumerate}
    \item Let $M$ be an $n$-dimensional connected manifold with a smooth action by a Lie group $G$. The action of $G$ is
    called a \emph{cohomogeneity-one action} if there exists an orbit with dimension $n-1$.

    \item An orbit $\mathcal{O}$ of a cohomogeneity-one action is called a \emph{principal orbit} if there is an open subset
    $U$ of $M$ with the following properties: $\mathcal{O}\subseteq U$ and $U$ is $G$-equivariantly diffeomorphic to
    $\mathcal{O}\times (-\epsilon,\epsilon)$, where $\epsilon>0$. It can be proven that this condition is equivalent to
    $\dim{\mathcal{O}} = n-1$. All principal orbits are $G$-equivariantly diffeomorphic to each other and the union of all
    principal orbits is an open dense subset of $M$.

     \item A Spin($7$)-manifold $(M,\Omega)$ is called \emph{of cohomogeneity one} if there exists a cohomogeneity-one
     action on $M$ which preserves $\Omega$ (and thus the associated metric).
\end{enumerate}
\end{DaL}

Since any Ricci-flat homogeneous metric is flat (see Alekseevskii,
Kimelfeld \cite{Alki}), spaces of cohomogeneity one are the most
symmetric manifolds which may admit metrics with exceptional
holonomy. For the following considerations, we fix some notation.

\begin{Conv} \label{CohomOneConv}
Let $G$ be a compact Lie group which acts with cohomogeneity one
on a Spin($7$)-manifold $(M,\Omega)$. The associated metric on $M$
we denote by $g$. We identify any orbit of $G$ with the quotient
of $G$ by the isotropy group. The principal orbit shall be $G/H$
and $G/K$ shall be a non-principal orbit. The union of all
principal orbits will be denoted by $M^0$. After conjugation, $H$
has to be a subgroup of $K$. The Lie algebras of $G$, $H$, and
$K$, we denote by $\mathfrak{g}$, $\mathfrak{h}$, and
$\mathfrak{k}$. Let $q$ be an auxiliary $\text{Ad}_K$-invariant
metric on $\mathfrak{g}$. We identify the tangent space of $G/H$
with the $q$-orthogonal complement $\mathfrak{m}$ of
$\mathfrak{h}$ in $\mathfrak{g}$. The tangent space of $G/K$ can
be identified with the complement $\mathfrak{p}$ of
$\mathfrak{k}$. We denote the normal space of the orbit $G/K$ by
$\mathfrak{p}^\perp$. On any cohomogeneity-one manifold, there
exists a geodesic which intersects all orbits perpendicularly. We
fix such a geodesic $\gamma$ and parameterize it by arclength. The
parameter of $\gamma$ we denote by $t$.
\end{Conv}

The following theorem of Mostert \cite{Mostert} gives us some information on the shape of $M$.

\begin{Th}
\begin{enumerate}
    \item If $G$ acts isometrically on a Riemannian cohomoge\-neity-one manifold $(M,g)$, $M/G$ is homeomorphic to the circle
    $S^1$, $[0,1]$, $[0,\infty)$, or $\mathbb{R}$. The inner points of $M/G$ correspond to principal orbits and the endpoints
    of the intervals to non-principal orbits.
    \item Let $G/H$ be a principal and $G/K$ be a non-principal orbit. The quotient $K/H$ is a sphere.
    \item Any sufficiently small tubular parameterize of the non-principal orbit $G/K$ is a disc bundle over
    $G/K$. The projection map maps a point $gH$  of the principal orbit to $gK$.
\end{enumerate}
\end{Th}

\begin{Rem}
$\text{Spin}(7)$-manifolds with certain kinds of singularities are
in issue in M-theory (see Acharya, Gukov \cite{Ach}). Therefore,
we also consider the cases where $M$ is not a manifold but an
orbifold. If $K/H$ is a quotient of a sphere by a discrete group
$\Gamma$, the tubular parameterize of $G/K$ is an
$\mathbb{R}^{\dim{K/H}+1}/\Gamma$-bundle over $G/K$ and $M$ thus
is an orbifold. In this section, we state our theorems for
manifolds only. Nevertheless, it is easily possible to adapt them
to the orbifold-case.
\end{Rem}

Since the volume of the metric on $K/H$ shrinks to zero as we approach the singular orbit, we will refer to $K/H$ as the
\emph{collapsing sphere} or if $\dim{K/H}=1$, as the \emph{collapsing circle}. We can restrict the topology of $M$ even
further. If $M/G=\mathbb{R}$, $(M,g)$ contains a complete geodesic which minimizes the length between any of its points.
It follows from the Cheeger-Gromoll splitting theorem that $M$ is a Riemannian product of $\mathbb{R}$ and a
seven-dimensional manifold. Since then the holonomy would be a subgroup of $G_2$, we will not consider this case. If
$M/G$ was $S^1$, the universal cover $\widetilde{M}$ would satisfy $\widetilde{M}/G=\mathbb{R}$. Therefore, we
exclude that case, too. If $M$ had two non-principal orbits, it would be compact.
Since it is Ricci-flat, all Killing vector fields are parallel and commute with each other. $G/H$ thus is a flat torus. It is
easy to see that $M$ has to be flat, too.

There are two kinds of non-principal orbits. If $K/H=S^0=\mathbb{Z}_2$, the orbit $G/K$ is called an \emph{exceptional orbit}.
Otherwise, it is a \emph{singular orbit}. If there is exactly one exceptional orbit, $M$ would be twofold covered by a space
$\widetilde{M}$ with $\widetilde{M}/G=\mathbb{R}$. Motivated by the above considerations, we assume from now on that
there is exactly one singular orbit and all other orbits are principal.

On any principal orbit, there exists a canonical $G_2$-structure $\omega$ which is related to $\Omega$ by
$\Omega:= \ast\omega + dt\wedge\omega$. The equation $d\Omega=0$ can be written in terms of $\omega$.

\begin{Th}
\label{HitchinThm} Let $G/H$ be a seven-dimensional homogeneous
space and $\omega$ be a $G$-invariant cocalibrated
$G_2$-structure on $G/H$. Then there exists an $\epsilon>0$ and a
one-parameter family $(\omega_t)_{t\in(-\epsilon,\epsilon)}$ of
$G$-invariant $G_2$-structures on $G/H$ such that the initial
value problem

\begin{eqnarray}
\label{Evol} \frac{\partial}{\partial t}\ast_{\scriptscriptstyle G/H} \omega_t & = &
d_{\scriptscriptstyle G/H}\omega_t\\ \label{InitialEvol} \omega_0
& = & \omega
\end{eqnarray}

has a unique solution on $G/H\times(-\epsilon,\epsilon)$. In the above formula, $\tfrac{\partial}{\partial t}$ denotes the Lie
derivative in $t$-direction. The index $G/H$ of $d$ and $\ast$ emphasizes that we consider the exterior derivative on $G/H$ instead of
$G/H\times (-\epsilon,\epsilon)$. If $\epsilon$ is sufficiently small, $\omega_t$ is for all $t\in (-\epsilon,\epsilon)$
a $G_2$-structure and we have $d_{\scriptscriptstyle G/H}\ast_{\scriptscriptstyle G/H} \omega_t=0$. The
four-form $\Omega:=\ast_{\scriptscriptstyle G/H}\omega+dt \wedge\omega$ is a $G$-invariant parallel
$\text{Spin}(7)$-structure on $G/H \times(-\epsilon,\epsilon)$.

Conversely, let $\Omega$ be a parallel Spin($7$)-structure
preserved by a cohomogeneity-one action of a Lie group $G$. We
identify the union of all principal orbits $G$-equivariantly with
$G/H\times I$, where the metric on $I$ is $dt^2$. In this
situation, the $G_2$-structures on the principal orbits are
cocalibrated and satisfy equation (\ref{Evol}).
\end{Th}

\begin{Rem}
\begin{enumerate}
    \item The above theorem was proven by Hitchin \cite{Hitch} for the more general case, where $\omega$ is a
    (not necessarily homogeneous) cocalibrated $G_2$-structure on a compact manifold.
    \item If $\omega$ is nearly parallel, the maximal solution of (\ref{Evol}) describes a cone over $G/H$.
    \item Since $(M,\Omega)$ is of cohomogeneity one, the equation (\ref{Evol}) is equivalent to a system of ordinary
    differential equations. In order to ensure that $M$ has the desired topology, we fix the
    initial conditions at the singular orbit $G/K$.
\end{enumerate}
\end{Rem}

Before we investigate the equation (\ref{Evol}), we have to choose the principal orbit. The following lemma answers the
question if $G/H$ admits a $G$-invariant $G_2$-structure.

\begin{Le} \label{G2HomLemma} (See \cite{Rei1}.)
Let $G/H$ be a homogeneous space such that $G$ acts effectively on $G/H$. Furthermore, let $p\in G/H$ be arbitrary. We
identify $H$ with its isotropy representation on $T_p G/H$ and $G_2$ with its seven-dimensional irreducible representation.
$G/H$ admits a $G$-invariant $G_2$-struc\-ture if and only if there exists a vector space isomorphism $\varphi: T_p G/H
\rightarrow \mathbb{R}^7$ such that $\varphi H \varphi^{-1} \subseteq G_2$.
\end{Le}

Our next step is to describe the space of all $G$-invariant
$G_2$-structures on $G/H$ explicitly. This can be done in two
steps. Any $G$-invariant metric on $G/H$ can be identified via $q$
with an $H$-equivariant endomorphism of $\mathfrak{m}$. These
endomorphisms can be classified with the help of Schur's lemma. Since
any manifold which admits a $G_2$-structure is orientable and
an $SO(7)$-structure is the same as a metric and an orientation,
we have  classified all $G$-invariant $SO(7)$-structures on $G/H$. The
classification of all $G$-invariant $G_2$-structures whose
extension to an $SO(7)$-structure is fixed, can be done with the help
of the following lemma.

\begin{Le} \label{G2StrucSpace} (See \cite{Rei}.) Let $G/H$ be a seven-dimensional homogeneous space. We assume that $G$ acts effectively and
that $G/H$ admits a $G$-invariant $G_2$-structure. Let $\mathcal{G}$ be an arbitrary $G$-invariant $SO(7)$-structure on $G/H$.
The space of all $G$-invariant $G_2$-structures on $G/H$ whose extension to an $SO(7)$-structure is $\mathcal{G}$ is
diffeomorphic to

\begin{equation}
\text{Norm}_{SO(7)}H/\text{Norm}_{G_2}H\:.
\end{equation}

In the above formula, $H$ is identified with its isotropy representation and $G_2$ and $SO(7)$
with their seven-dimensional irreducible representation. The \emph{normalizer}
$\text{Norm}_L L'$ of a subgroup $L'\subseteq L$ is defined as $\{g\in L| gL'g^{-1}=L'\}$.
\end{Le}

After having determined the space of all $G$-invariant
$G_2$-structures $\omega$ on $G/H$, we calculate $d\ast\omega$ and
have a description of the space of all cocalibrated invariant
$G_2$-structures. In some cases, we will not be able to describe
that space explicitly. In order to find examples of parallel
cohomogeneity-one $\text{Spin}(7)$-structures, it suffices to
construct a space of cocalibrated $G_2$-structures which is
invariant under equation (\ref{Evol}).

Not any solution of (\ref{Evol}) corresponds to a metric with
holonomy $\text{Spin}(7)$. The reason for this is that $\Omega$
has not automatically a smooth extension to the singular orbit.
This is the case only if certain smoothness conditions are satisfied,
which we will describe in detail.

We split the tangent space of $M$ at a point $p\in G/K$
into the $K$-modules $\mathfrak{p}$ and $\mathfrak{p}^\perp$.
The orbits of the $K$-action on $\mathfrak{p}^\perp$ except $\{0\}$
are spheres of type $K/H$. Let $\mathcal{B}$ be a vector bundle
over the union of all principal orbits which admits a $K$-action
on the fibers. For reasons of simplicity we assume that there exist
non-negative numbers $s_1$ and $s_2$ such that the fibers of
$\mathcal{B}$ are contained in $\bigotimes^{s_1} TM \otimes
\bigotimes^{s_2} T^\ast M$. Moreover, let $\rho$ be a
$G$-invariant section of $\mathcal{B}$. Since $G/K$ is
homogeneous, $\rho$ is determined by its values at $p$.

Let $\gamma$  be a geodesic which intersects all orbits
perpendicularly. We assume that $\gamma(0)\in G/K$. Since the
action of $G$ on $\gamma$ generates all of $M$, it suffices to
consider $\rho$ along $\gamma$ only. The metric $g$ is Ricci-flat.
It is well-known (see DeTurck, Kazdan \cite{DeKa}) that any
Einstein metric is analytic. We therefore assume that $\rho$ is a
power series with respect to $t$. The $m^{th}$ derivative of
$\rho$ in the vertical direction can be considered as a map, which
assigns to a tuple $(v_1,\ldots,v_m)\in \mathfrak{p}^\perp$ an
element of the fiber $\mathcal{B}_p$. This map can be extended to
a map $S^m(\mathfrak{p}^\perp) \rightarrow \mathcal{B}_p$, where
$S^m(\mathfrak{p}^\perp)$ denotes the $m^{th}$ symmetric power of
$\mathfrak{p}^\perp$. Since $\rho$ is analytic, the sequence of
those maps determines $\rho$. If $\rho$ has a smooth extension to
the singular orbit, the above maps are $K$-equivariant.
Conversely, we have

\begin{Th} \label{SmoothExtension} (See Eschenburg, Wang \cite{Esch}.)
Let $(M,g)$ be Riemannian manifold with an isometric action of
cohomogeneity one by a Lie group $G$. We assume that there is a
singular orbit $G/K$. Let $\mathcal{B}\subseteq \bigotimes^{s_1}
TM\otimes \bigotimes^{s_2} T^\ast M$ be a vector bundle over $M$
whose fibers at the singular orbit are $K$-equivariantly
isomorphic to a $K$-module $B$. Let $r:(0,\varepsilon)\rightarrow
B$, where $\varepsilon>0$, be a real analytic map with Taylor
expansion $\sum_{m=1}^\infty r_m t^m$. We can identify $r$ with a
tensor field $\rho$ along a geodesic $\gamma$ which intersects all
orbits perpendicularly. By the action of $G$, we can extend $\rho$
to the union of all principal orbits. $\rho$ is well-defined and
has a smooth extension to the singular orbit if and only if

\begin{equation}
r_m\in \imath_m(W_m)\quad\forall m\in\mathbb{N}_0\:.
\end{equation}

In the above formula, $W_m$ denotes the space of all $K$-equivariant maps

\begin{equation}
W_m:=\{P:S^m(\mathfrak{p}^\perp)\rightarrow B |P\:\text{is linear and $K$-equivariant}\}
\end{equation}

and $\imath_m$ is the evaluation map

\begin{equation}
\begin{split}
\imath_m & :W_m\rightarrow B\\ \imath_m(P) & :=
P(\gamma'(0))\:.\\
\end{split}
\end{equation}
\end{Th}

In the following, we restrict ourselves to metrics with no "mixed coefficients", i.e.

\begin{equation}
\label{WEschAssumption1}
g\in S^2(\mathfrak{p}) \oplus S^2(\mathfrak{p}^\perp)\:.
\end{equation}

Instead of $W_m$, it suffices to study the spaces

\begin{equation}
\begin{split}
W^h_m & :=\{P:S^m(\mathfrak{p}^\perp)\rightarrow S^2(\mathfrak{p}) |P\:\text{is linear and $K$-equivariant}\}\quad\text{and} \\
W^v_m & :=\{P:S^m(\mathfrak{p}^\perp)\rightarrow S^2(\mathfrak{p}^\perp) |P\:\text{is linear and $K$-equivariant}\} \\
\end{split}
\end{equation}

in order to prove the smoothness. The reason for assumption (\ref{WEschAssumption1}) is that later on we need a result which
is proven only if (\ref{WEschAssumption1}) is satisfied.

We often write our metric as $g_t + dt^2$ where $g_t\in S^2(\mathfrak{m})$ is the restriction of $g$
to a principal orbit. Since $t$ can be considered as a distance function on $\mathfrak{p}^{\perp}$ and
$g_t|_{\mathfrak{p}^\perp\times\mathfrak{p}^\perp}$ describes the metric on the collapsing sphere,
$\mathfrak{p}^\perp$ is equipped with "polar" rather than "Euclidean" coordinates. We therefore have
to modify Theorem \ref{SmoothExtension} in order to fit our needs.

\begin{Rem}
\label{SmoothnessRemark}
\begin{enumerate}
    \item By the choice of our coordinates we have fixed $\|\tfrac{\partial}{\partial t}\|=1$ and $g(\tfrac{\partial}{
    \partial t},v)=0$ for all $v\in\mathfrak{m}$. The degrees of freedom for the higher derivatives of the vertical part of
    $g$ will therefore seem to be fewer as Theorem \ref{SmoothExtension} predicts.
    \item Let $v$ be a tangent vector of the collapsing sphere $K/H$. The metric on $K/H$ has to approach the round metric of
    a sphere of radius $t$. This condition fixes the value of $\tfrac{\partial}{\partial t}|_{t=0} g_t(v,v)^{\tfrac{1}{2}}$.
    If $K/H$ is a sphere, we can compute this value with the help of the fact that the length of any great circle on $K/H$ has to be
    $2\pi t + O(t^2)$ for $t\rightarrow 0$. If $K/H$ is a quotient of a sphere by a discrete group, we can use the estimate
    $\tfrac{1}{t} + O(1)$ for the sectional curvature. The above statements are in fact equivalent to the smoothness condition
    of $0^{th}$ order for the vertical part.
    \item Since the length of $v$ shrinks to zero, any statement on the $m^{th}$ derivative of $g$ in the vertical direction
    translates into a statement on $\tfrac{\partial^m}{\partial t^m}|_{t=0} \tfrac{1}{t} g_t(v,v)^{\tfrac{1}{2}}$.  Because of
    l'H\^{o}pital's rule this is essentially a statement on the $(m+1)^{st}$ derivative of $g_t$.
\end{enumerate}
\end{Rem}

We assume that $(M^0,\Omega)$ is of holonomy Spin($7$) and that the metric $g$ has a smooth extension to the singular orbit.
In this situation, the holonomy of $(M,g)$ equals Spin($7$), too. Therefore, there exists a unique smooth Spin($7$)-structure
$\widetilde{\Omega}$ on $M$. Without loss of generality, we can assume that $\Omega$
and $\widetilde{\Omega}$ coincide on $M^0$. This observation proves that $\widetilde{\Omega}$ is a smooth extension of
$\Omega$ to the singular orbit and we do not have to prove the smoothness conditions for $\Omega$.

If the holonomy $\text{Hol}$ is a smaller group, for example $Sp(2)$ or $SU(4)$, we can prove by
similar arguments that there exists a smooth $\text{Hol}$-structure on $M$. Since $G$ acts by isometries, it leaves the
holonomy bundle invariant and the $\text{Hol}$-structure thus is $G$-invariant.

Since $\dim{G/K} < \dim{G/H}$, the equation (\ref{Evol}) sometimes degenerates at the singular orbit. More precisely, it
is equivalent to a system which contains equations of type $c'(t)= \ldots + \tfrac{a(t)}{b(t)} + \ldots$ with
$\lim_{t\rightarrow 0} a(t) = \lim_{t\rightarrow 0} b(t) = 0$. In that situation, we cannot apply the theorem of
Picard-Lindel\"of, since the right-hand side $f(a,b,c,\ldots)$ is not defined on an open set. There are indeed cases,
where the solution of our initial value problem depends on initial conditions of higher order which can be chosen freely.
In order to classify the solutions of (\ref{Evol}), we make the power series ansatz

\begin{equation}
\omega = \sum_{m=0}^\infty \omega_m t^m\quad\text{with}\quad \omega_m\in{\bigwedge}^3 \mathfrak{m}\quad
\text{$\text{Ad}_H$-invariant}\:.
\end{equation}

In the cases which we will consider, we fix for any choice of the
metric $g_t$ on $G/H$ a single cocalibrated $G_2$-structure $\omega_t$
whose associated metric is $g_t$. The cocalibrated $G_2$-structures with that
property are in many cases a discrete set. The set of $G_2$-structures
which is obtained from a sufficiently large set of $g_t$ is preserved by
(\ref{Evol}) and our restriction to those $G_2$-structures thus is justified.
$\omega_t$ will always depend analytically on $g_t$ and
the cohomogeneity-one metric $g$ is Ricci-flat. Since any Einstein metric is analytic
(see DeTurck, Kazdan \cite{DeKa}), we are allowed to make the
above power series ansatz. Equation (\ref{Evol})
yields the following system of recursive equations for the $g_m$

\begin{equation}
\label{RekSys}
\mathcal{L}_m(g_m) = P_m(g_0,\ldots,g_{m-1})\:.
\end{equation}

$\mathcal{L}_m$ is a linear operator acting on the space of all $\text{Ad}_H$-invariant
symmetric bilinear forms on $\mathfrak{m}$ and $P_m$ is a polynomial. For a fixed
choice of the principal and the singular orbit, $\mathcal{L}_m$ can be
calculated for all $m$. We will see that in each of our cases there exists an $m_0\in
\mathbb{N}$ such that for all $m\geq m_0$ $\mathcal{L}_m$ is always
invertible. It follows from Theorem \ref{EinsteinCohom1Solution}, which we will state
below, that there is a deeper reason behind this. By solving (\ref{RekSys}) for all
$m<m_0$, we can classify all formal power series which solve equation
(\ref{Evol}). By certain arguments which we will make explicit when we need them
we can check the smoothness conditions. All which is left to be done is to check if
the power series converges. This follows by a theorem of Eschenburg and Wang
\cite{Esch} on cohomogeneity-one Einstein metrics. Before we state that theorem,
we make the following assumption.

\begin{Assumpt}
\label{WEschAssumpt}
The tangent space $\mathfrak{p}$ and the normal space $\mathfrak{p}^\perp$ of the singular orbit shall have no $H$-submodule
of positive dimension in common.
\end{Assumpt}

\begin{Th}\label{EinsteinCohom1Solution} (See Eschenburg, Wang \cite{Esch}.) Let $M$ be a manifold equipped with a
cohomogeneity-one action by a compact Lie group $G$. We assume that the principal orbits of this action are $G$-equivariantly
diffeomorphic to $G/H$ and that there is a singular orbit $G/K$. Moreover, we assume that \ref{WEschAssumpt} is satisfied.

Let $g_0$ be an arbitrary $G$-invariant metric on the singular orbit. Furthermore, let $g'_0:\mathfrak{p}^\perp
\rightarrow S^2(\mathfrak{p})$ be a linear, $K$-equivariant map. Finally, let $\lambda\in\mathbb{R}$ be arbitrary. In this
situation, there exists a $G$-invariant Einstein metric $g$ on a sufficiently small tubular neighborhood of the singular orbit
which has the following properties.

\begin{enumerate}
    \item $g$ has $\lambda$ as Einstein constant.
    \item The restriction of $g$ to the singular orbit is $g_0$.
    \item The first derivation of $g$ at the singular orbit in the normal directions is $g'_0$.
\end{enumerate}

The set of all Einstein metrics with the above properties depends on additional initial conditions of higher order, which we can prescribe
arbitrarily. The freedom for the $m^{th}$ derivative of the metric in the horizontal or vertical direction can be described by

\begin{equation}
\begin{array}{ll}
W^h_m / W^h_{m-2} &
\quad\text{in the horizontal case if $m\geq 2$}\\ &\\ W_2^v / W_0^v
& \quad\text{in the vertical case if $m=2$}\\
\end{array}
\end{equation}

and there are no further free parameters in the vertical direction.
\end{Th}

\begin{Rem}
\label{WEschThmRem}
\begin{enumerate}
    \item Up to constant multiples, there is exactly one $K$-invariant scalar product $h$ on $\mathfrak{p}^\perp$. The reason
    for this is that the orbits of the $K$-action on $\mathfrak{p}^\perp$ are spheres. In particular, we have $\dim{W_0^v}=1$.
    \item $S^m(\mathfrak{p}^{\perp})$ is embedded into $S^{m+2}(\mathfrak{p}^{\perp})$
    by the map $\imath$ with $\imath(P):= h \vee P$, where $\vee$ is the symmetrized tensor product. $\imath$ induces
    canonical embeddings of $W^h_{m-2}$ and $W^v_{m-2}$ into $W^h_m$ and $W^v_m$, which we have implicitly used in
    the formulation of the above theorem.
    \item Since we assume that \ref{WEschAssumpt} is satisfied, the metric is automatically
    contained in $S^2(\mathfrak{p}) \oplus S^2(\mathfrak{p}^\perp)$. Analogously, we have $W_m = W_m^h \oplus W_m^v$.
    \item For sufficiently large $m$ the chains $W_0 \subseteq W_2 \subseteq W_4 \subseteq \ldots$ and $W_1 \subseteq W_3
    \subseteq W_5 \subseteq \ldots$ stabilize. In particular, there are only finitely many initial conditions which we can
    prescribe.
    \item \label{ConvRem} The cohomogeneity-one Einstein condition is system of second order differential equations. That
    system yields a recursive equation which is similar to (\ref{RekSys}). The convergence of the power series solutions can
    be shown with the help of a Picard iteration. More precisely, any power series which satisfies the Einstein condition
    converges if $\mathfrak{p}$ and $\mathfrak{p}^\perp$ have no $H$-submodule in common. In particular, any formal power
    series solution of (\ref{RekSys}) converges if we assume \ref{WEschAssumpt}. In some of the cases which we will consider,
    this assumption is not satisfied. Nevertheless, the convergence can be proven in those cases, too. In the article, we
    restrict ourselves to metrics which are diagonal with respect to a fixed basis of $\mathfrak{m}$. The metric is therefore
    an element of $S^2(\mathfrak{p}) \oplus S^2(\mathfrak{p}^\perp)$. Moreover, the equations for the Einstein condition do
    not change a diagonal metric into a non-diagonal one. The space $S^2(\mathfrak{p}) \oplus
    S^2(\mathfrak{p}^\perp)$ thus is invariant under the Picard iteration. This allows us to repeat the arguments of
    \cite{Esch} and to prove the convergence.
    \item In order to deduce the smoothness conditions, we have to describe the spaces $W_m^h$ and $W_2^v$. We can therefore
    easily apply Theorem \ref{EinsteinCohom1Solution} and obtain new examples of cohomogeneity-one Einstein metrics.
    \item The above theorem predicts that certain second derivatives in the vertical direction can be chosen freely. For
    similar reasons as in Remark \ref{SmoothnessRemark}, these derivatives are third derivatives with respect to $t$.
    As we have already remarked in \ref{SmoothnessRemark}, we have $g(\tfrac{\partial}{\partial t},v)=0$ for all
    $v\in\mathfrak{m}$ by the choice of our coordinates. We therefore have to ignore the free parameters in the vertical
    direction which describe the change of $g(\tfrac{\partial}{\partial t},v)$.
    \item In general, the power series converges for small values of $t$ only. The metrics which we construct with the help
    of the above theorem are thus incomplete. They can be extended to complete metrics if and only if the equation
    $\frac{\partial}{\partial t}\ast\omega = d\omega$ or $\text{Ric} = \lambda g$ has a solution for all $t\in[0,\infty)$ or
    there are two singular orbits.
\end{enumerate}
\end{Rem}

After we have proven the convergence, we have constructed metrics
whose holonomy group is a subgroup of Spin($7$) on a tubular
parameterize of the singular orbit. In order to decide if the
holonomy is all of Spin($7$), we need the following lemma.

\begin{Le} \label{SU4Lemma} (See \cite{Rei}.)
\begin{enumerate}
    \item Let $M$ be an eight-dimensional manifold which carries a parallel $SU(4)$-structure $\mathfrak{G}$. We denote the
    space of all parallel $\text{Spin}(7)$-structures on $M$ which are an extension of $\mathfrak{G}$ and have the same
    extension to an $SO(8)$-structure as $\mathfrak{G}$ by $\mathcal{S}$. Any connected component of $\mathcal{S}$ is
    diffeomorphic to a circle.
    \item Let $M$ be an eight-dimensional manifold which carries a one-parame\-ter family $\mathcal{S}$ of parallel
    $\text{Spin}(7)$-structures. Moreover, let the extension of the $\text{Spin}(7)$-structures to an $SO(8)$-structure always
    be the same and let $\mathcal{S}$ be diffeomorphic to a circle. Then, there also exists a parallel $SU(4)$-structure on $M$.
\end{enumerate}
\end{Le}

With the help of the facts which we have collected in this section we are now able to construct examples of Spin($7$)-manifolds.

\section{The geometry of the Aloff-Wallach spaces}

Before we construct cohomogeneity-one metrics with an Aloff-Wallach space as the principal orbit, we have to study those spaces in
detail. The Aloff-Wallach spaces are certain homogeneous spaces which were introduced by Aloff and Wallach \cite{Al}
in order to study metrics with positive sectional curvature. Let

\begin{equation}
\begin{aligned}
i_{k,l}:U(1) & \rightarrow SU(3)\\
i_{k,l}\left(e^{i\varphi}\right) & := \left(\,\begin{array}{ccc}
e^{ik\varphi} & 0 & 0\\ 0 & e^{il\varphi} & 0\\ 0 & 0 &
e^{-i(k+l)\varphi}\\
\end{array}\,\right)\quad \text{with}\:k,l\in\mathbb{Z}\:.\\
\end{aligned}
\end{equation}

We denote the image of $U(1)$ with respect to $i_{k,l}$ by
$U(1)_{k,l}$. Any one-dimensional subgroup of $SU(3)$ is conjugate
to a $U(1)_{k,l}$. The Aloff-Wallach space $N^{k,l}$ is defined as
the quotient $SU(3)/U(1)_{k,l}$. Without loss of generality, we
assume that $k$ and $l$ are coprime. Let $\sigma$ be a permutation
of the triple $(k,l,-k-l)$. It is easy to see that the spaces
$N^{k,l}$ and $N^{\sigma(k),\sigma(l)}$ are $SU(3)$-equivariantly
diffeomorphic. Let $2\mathfrak{u}(1)$ be the Cartan subalgebra of all
diagonal matrices in $\mathfrak{su}(3)$ and let
$\mathfrak{u}(1)_{k,l}$ be the Lie algebra of $U(1)_{k,l}$.
The Weyl group of $\mathfrak{su}(3)$ is isomorphic to the
permutation group $S_3$ and acts on $2\mathfrak{u}(1)$.
The action of a $\sigma\in S_3$ on $2\mathfrak{u}(1)$
changes $\mathfrak{u}(1)_{k,l}$ into
$\mathfrak{u}(1)_{\sigma(k),\sigma(l)}$. Therefore, the
$S_3$-action on $(k,l,-k-l)$ can be identified with the action
of the Weyl group. Since $N^{k,l}$ and $N^{-k,-l}$ are the same
manifold, we introduce the following convention.

\begin{Conv} \label{AWConv}
When we consider an Aloff-Wallach space $N^{k,l}$, we assume that $k\geq l\geq 0$. Later on, we will turn to another
convention which will be introduced at that point.
\end{Conv}

The Aloff-Wallach spaces $N^{1,0}$ or $N^{1,1}$ are called \textit{exceptional} and the other ones are called \textit{generic}.
Since there are many differences between the exceptional and the generic Aloff-Wallach spaces, we often
have to tread $N^{1,0}$, $N^{1,1}$, and the generic $N^{k,l}$ as separate cases. There are infinitely many homotopy types of
Aloff-Wallach spaces. This follows from the fact that

\begin{equation}
H^4(N^{k,l},\mathbb{Z})=\mathbb{Z}_{k^2+lk+l^2}\:.
\end{equation}

Some of the $N^{k,l}$ are homeomorphic but not diffeomorphic to each other. Examples of this fact can be found in Kreck, Stolz
\cite{Kreck}. On the Aloff-Wallach spaces $N^{k,l}$ there exist two nearly parallel $G_2$-structures, which depend on $k$ and
$l$ \cite{Cve}. The holonomy of the cones over those $G_2$-structures therefore is contained in
Spin($7$). We will see that the Aloff-Wallach spaces are not covered by a sphere. Therefore, the cones have a singularity at
the tip, which is not an orbifold singularity.

Our next step is to describe all $SU(3)$-invariant metrics on the Aloff-Wallach spaces. In order to do this, we fix the following basis
$(e_i)_{1\leq i\leq 8}$ of $\mathfrak{su}(3)$.

\begin{equation}
\label{qbasisAW}
\begin{array}{l}
\begin{array}{lll}
e_1:= E_1^2 - E_2^1 & e_2:=iE_1^2 + iE_2^1 & e_3:=E_1^3 - E_3^1 \\
&& \\
e_4:=iE_1^3 + iE_3^1 & e_5:= E_2^3 - E_3^2 & e_6:=iE_2^3 + iE_3^2 \\
\end{array} \\ \\
\:\: e_7:= (2l+k)iE_1^1 + (-2k-l)iE_2^2 + (k-l)iE_3^3 \\ \\
\:\: e_8:= kiE_1^1 + liE_2^2 - (k+l)iE_3^3\\
\end{array}
\end{equation}

$E_i^j$ denotes the $3\times 3$-matrix with a $1$ in the $i^{th}$
row and $j^{th}$ column and zeroes elsewhere. The Lie algebra
$\mathfrak{u}(1)_{k,l}$ is generated by $e_8$.
$q(X,Y):=-\text{tr}(XY)$ defines a biinvariant metric on
$\mathfrak{su}(3)$ and $(e_1,\ldots,e_7)$ is a basis of the
$q$-orthogonal complement $\mathfrak{m}$ of
$\mathfrak{u}(1)_{k,l}$. The isotropy action of
$\mathfrak{u}(1)_{k,l}$ splits $\mathfrak{m}$ into the following
irreducible submodules.

\begin{equation}
\label{V1V2V3}
\begin{array}{ll}
V_1 :=\text{span}(e_1,e_2) & V_2 :=\text{span}(e_3,e_4) \\
V_3 :=\text{span}(e_5,e_6) & V_4 :=\text{span}(e_7) \\
\end{array}
\end{equation}

The weights of the first three submodules are

\begin{equation}
\label{WeightsU1kl}
k-l,\: 2k+l,\: k+2l\:.
\end{equation}

If $N^{k,l}$ is generic, $V_1$, $V_2$, and $V_3$ are pairwise
inequivalent. If $(k,l)=(1,0)$, $V_1$ and $V_3$ are equivalent
and $V_2$ is not equivalent to the two other modules. In the case
where $k=l=1$, $V_1$ is trivial and $V_2$ and $V_3$ are equivalent
to each other. Any $SU(3)$-invariant metric $g$ on $N^{k,l}$ can
be identified via $q$ with a $\mathfrak{u}(1)_{k,l}$-equivariant
endomorphism of $\mathfrak{m}$. We can therefore classify the
invariant metrics with the help of Schur's lemma. If $N^{k,l}$ is
generic, the matrix representation $g_{ij}:=g(e_i,e_j)$ of $g$ is

\begin{equation}
\label{MetricsNkl} \left(\,\begin{array}{ccccccc} \hhline{--~~~~~}
\multicolumn{1}{|c}{a^2} & \multicolumn{1}{c|}{0} &&&&&\\
\multicolumn{1}{|c}{0} & \multicolumn{1}{c|}{a^2} &&&&&\\
\hhline{----~~~} && \multicolumn{1}{|c}{b^2} &
\multicolumn{1}{c|}{0} &&&\\ && \multicolumn{1}{|c}{0} &
\multicolumn{1}{c|}{b^2} &&&\\ \hhline{~~----~} &&&&
\multicolumn{1}{|c}{c^2} & \multicolumn{1}{c|}{0} &\\ &&&&
\multicolumn{1}{|c}{0} & \multicolumn{1}{c|}{c^2} &\\
\hhline{~~~~---} &&&&&& \multicolumn{1}{|c|}{f^2}\\
\hhline{~~~~~~-}
\end{array}\,\right)\quad\text{with}\:\:a,b,c,f\in \mathbb{R}\setminus\{0\}\:.
\end{equation}

The coefficients of the $\text{Spin}(7)$-structure which we will construct contain odd powers of $a$, $b$, $c$, and $f$.
Therefore, we allow these numbers to be negative, too, although this does not change the metric. Next, we assume that $k=1$
and $l=0$. The matrix representation of $g$ with respect to the basis $(e_1,e_2,e_5,e_6,e_3,e_4,e_7)$ is

\begin{equation}
\label{MetricsN10} \left(\,\begin{array}{ccccccc} \hhline{----~~~}
\multicolumn{1}{|c}{a^2} & \multicolumn{1}{c|}{0} & \beta_{1,5} &
\multicolumn{1}{c|}{\beta_{1,6}} &&&\\ \multicolumn{1}{|c}{0} &
\multicolumn{1}{c|}{a^2} & -\beta_{1,6} &
\multicolumn{1}{c|}{\beta_{1,5}} &&&\\ \hhline{----~~~}
\multicolumn{1}{|c}{\beta_{1,5}} &
\multicolumn{1}{c|}{-\beta_{1,6}} & c^2 & \multicolumn{1}{c|}{0}
&&&\\ \multicolumn{1}{|c}{\beta_{1,6}} &
\multicolumn{1}{c|}{\beta_{1,5}} & 0 & \multicolumn{1}{c|}{c^2}
&\\ \hhline{------~} &&&& \multicolumn{1}{|c}{b^2} &
\multicolumn{1}{c|}{0} &\\ &&&& \multicolumn{1}{|c}{0} &
\multicolumn{1}{c|}{b^2} &\\ \hhline{~~~~---} &&&&&&
\multicolumn{1}{|c|}{f^2}\\ \hhline{~~~~~~-}
\end{array}\,\right)
\end{equation}

with $a,b,c,f,\beta_{1,5},\beta_{1,6}\in \mathbb{R}$, $a^2c^2\geq \beta_{1,5}^2+\beta_{1,6}^2$, $b\neq 0$, and $f\neq 0$. If
$k=l=1$, the matrix representation of $g$ with respect to $(e_1,e_2,e_7,e_3,e_4,e_5,e_6)$ is

\begin{equation}
\label{MetricsN11} \left(\,\begin{array}{ccccccc} \hhline{---~~~~}
\multicolumn{1}{|c}{a_1^2} & \beta_{1,2} &
\multicolumn{1}{c|}{\beta_{1,7}} &&&&\\
\multicolumn{1}{|c}{\beta_{1,2}} & a_2^2 &
\multicolumn{1}{c|}{\beta_{2,7}} &&&&\\
\multicolumn{1}{|c}{\beta_{1,7}} & \beta_{2,7} &
\multicolumn{1}{c|}{f^2} &&&&\\ \hhline{-------} &&&
\multicolumn{1}{|c}{b^2} & \multicolumn{1}{c|}{0} &
\multicolumn{1}{|c}{ \beta_{3,5}} &
\multicolumn{1}{c|}{\beta_{3,6}}\\ &&& \multicolumn{1}{|c}{0} &
\multicolumn{1}{c|}{b^2} & \multicolumn{1}{|c}{ -\beta_{3,6}} &
\multicolumn{1}{c|}{\beta_{3,5}}\\ \hhline{~~~----} &&&
\multicolumn{1}{|c}{\beta_{3,5}} &
\multicolumn{1}{c|}{-\beta_{3,6}} & \multicolumn{1}{|c}{c^2} &
\multicolumn{1}{c|}{0}\\ &&& \multicolumn{1}{|c}{\beta_{3,6}}&
\multicolumn{1}{c|}{\beta_{3,5}} & \multicolumn{1}{|c}{0} &
\multicolumn{1}{c|}{c^2}\\ \hhline{~~~----}
\end{array}\,\right)
\end{equation}

As in the other cases, the above matrix has to be positive definite. Throughout the article we assume that the metric on the
principal orbit is diagonal with respect to $(e_1,\ldots,e_7)$. This assumption simplifies our calculations and we
nevertheless obtain interesting results. Moreover, some of the non-diagonal metrics can be changed by the action of the
normalizer $\text{Norm}_{SU(3)} U(1)_{k,l}$ into diagonal ones.

In Eschenburg, Wang \cite{Esch}, Grove, Ziller \cite{Gro}, and
Schwachh\"ofer, Tuschmann \cite{Schw2} it is explained how the
Einstein condition $\text{Ric} = \lambda g$ for a
cohomogeneity-one manifold can be rewritten as a system of
ordinary differential equations for the coefficient functions of
$g$. The Ricci-tensor of a generic $N^{k,l}$ with an arbitrary
$SU(3)$-invariant metric was calculated by Wang \cite{Wang1}.
If we put the results of the above papers together, we see that the
Einstein condition for a generic $N^{k,l}$ is equivalent to

\begin{equation}
\label{EinsteinCondGeneric}
\begin{array}{rcl}
-\frac{a''}{a}+\frac{a'^2}{a^2}-\frac{a'}{a}\left(2\frac{a'}{a}
+2\frac{b'}{b}+2\frac{c'}{c}+\frac{f'}{f}\right)
+ \frac{6}{a^2} - \frac{1}{2}\frac{(k+l)^2}{(k^2+lk+l^2)^2}
\frac{f^2}{a^4} && \\
 + \frac{a^4-b^4-c^4}{a^2b^2c^2} & = & \lambda \\
 && \\
-\frac{b''}{b}+\frac{b'^2}{b^2}-\frac{b'}{b}\left(2\frac{a'}{a}
+2\frac{b'}{b}+2\frac{c'}{c}+\frac{f'}{f}\right)
+ \frac{6}{b^2} - \frac{1}{2}\frac{l^2}{(k^2+lk+l^2)^2}
\frac{f^2}{b^4} && \\
+ \frac{b^4-a^4-c^4}{a^2b^2c^2} & = & \lambda \\
&& \\
-\frac{c''}{c}+\frac{c'^2}{c^2}-\frac{c'}{c}\left(2\frac{a'}{a}
+2\frac{b'}{b}+2\frac{c'}{c}+\frac{f'}{f}\right)
+ \frac{6}{c^2} - \frac{1}{2}\frac{k^2}{(k^2+lk+l^2)^2}
\frac{f^2}{c^4} && \\
+ \frac{c^4-a^4-b^4}{a^2b^2c^2} & = & \lambda \\
&& \\
-\frac{f''}{f}+\frac{f'^2}{f^2}-\frac{f'}{f}\left(2\frac{a'}{a}
+2\frac{b'}{b}+2\frac{c'}{c}+\frac{f'}{f}\right)
+\frac{1}{2}\frac{(k+l)^2}{(k^2+lk+l^2)^2}\frac{f^2}{a^4}
&& \\
+\frac{1}{2}\frac{l^2}{(k^2+lk+l^2)^2}\frac{f^2}{b^4}
+\frac{1}{2}\frac{k^2}{(k^2+lk+l^2)^2}\frac{f^2}{c^4} & = &
\lambda \\
&& \\
 - 2\frac{a''}{a} - 2\frac{b''}{b} - 2\frac{c''}{c} -
\frac{f''}{f} & = & \lambda\\
\end{array}
\end{equation}

If the principal orbit is $N^{1,0}$ and carries a diagonal metric,
the Einstein condition is equivalent to
(\ref{EinsteinCondGeneric}) with $k=1$ and $l=0$. In particular, a
diagonal metric cannot be changed by the equations
(\ref{EinsteinCondGeneric}) into a non-diagonal one for a
different value of $t$. If $k=l=1$, we do not necessarily have
$g(e_1,e_1)=g(e_2,e_2)$ and $\text{Ric} = \lambda g$ becomes

\begin{equation}
\label{EinsteinCondExcept}
\begin{array}{rcl}
-\frac{a_1''}{a_1}+\frac{a_1'^2}{a_1^2}-\frac{a_1'}{a_1}\left(\frac{a_1'}{a_1}
+\frac{a_2'}{a_2}+2\frac{b'}{b}+2\frac{c'}{c}+\frac{f'}{f}\right)
+ \frac{6}{a_1^2} -\frac{2}{9}\frac{f^2}{a_1^2a_2^2} && \\
+ 18\frac{a_1^4-a_2^4}{a_1^2a_2^2f^2} +
\frac{a_1^4-b^4-c^4}{a_1^2b^2c^2} & = & \lambda\\
&& \\
-\frac{a_2''}{a_2}+\frac{a_2'^2}{a_2^2}-\frac{a_2'}{a_2}\left(\frac{a_1'}{a_1}
+\frac{a_2'}{a_2}+2\frac{b'}{b}+2\frac{c'}{c}+\frac{f'}{f}\right)
+ \frac{6}{a_2^2} -\frac{2}{9}\frac{f^2}{a_1^2a_2^2} && \\
 + 18\frac{a_2^4-a_1^4}{a_1^2 a_2^2 f^2} + \frac{a_2^4-b^4-c^4}{a_2^2
b^2 c^2} & = & \lambda\\ && \\
-\frac{b''}{b}+\frac{b'^2}{b^2}-\frac{b'}{b}\left(\frac{a_1'}{a_1}
+\frac{a_2'}{a_2}+2\frac{b'}{b}+2\frac{c'}{c}+\frac{f'}{f}\right)
+ \frac{6}{b^2} - \frac{1}{18}\frac{f^2}{b^4} && \\ +
\frac{b^4-a_1^4-c^4}{2a_1^2 b^2 c^2} + \frac{b^4-a_2^4-c^4}{2a_2^2
b^2 c^2} & = & \lambda\\ && \\
-\frac{c''}{c}+\frac{c'^2}{c^2}-\frac{c'}{c}\left(\frac{a_1'}{a_1}
+\frac{a_2'}{a_2}+2\frac{b'}{b}+2\frac{c'}{c}+\frac{f'}{f}\right)
+ \frac{6}{c^2} - \frac{1}{18}\frac{f^2}{c^4} && \\ +
\frac{c^4-a_1^4-b^4}{2a_1^2b^2c^2} +
\frac{c^4-a_2^4-b^4}{2a_2^2b^2c^2} & = & \lambda\\ && \\
-\frac{f''}{f}+\frac{f'^2}{f^2}-\frac{f'}{f}\left(\frac{a_1'}{a_1}
+\frac{a_2'}{a_2}+2\frac{b'}{b}+2\frac{c'}{c}+\frac{f'}{f}\right)
+ \frac{36}{f^2} - 18\frac{a_1^2}{a_2^2f^2}  -
18\frac{a_2^2}{a_1^2f^2} && \\ + \frac{2}{9}\frac{f^2}{a_1^2
a_2^2} + \frac{1}{18}\frac{f^2}{b^4} + \frac{1}{18}\frac{f^2}{c^4}
& = & \lambda\\ && \\ - \frac{a_1''}{a_1} - \frac{a_2''}{a_2} -
2\frac{b''}{b} - 2\frac{c''}{c} - \frac{f''}{f} & = & \lambda\\
\end{array}
\end{equation}

As in the previous case, the differential equations (\ref{EinsteinCondExcept}) preserve the space of all diagonal metrics on
$N^{k,l}$. According to Remark \ref{WEschThmRem}.\ref{ConvRem}, any formal power series which solves one of the above systems
and satisfies the smoothness conditions converges.

Our next step is to deduce a system of ordinary differential equations which is
equivalent to $d\Omega=0$. Let $M$ be a cohomogeneity-one manifold whose
principal orbit is a generic Aloff-Wallach space. The following basis of the tangent space
induces an $SU(3)$-invariant Spin($7$)-structure on $M$ whose associated metric restricted
to a principal orbit is (\ref{MetricsNkl}).

\begin{equation}
\label{CBasisNkl}
\begin{array}{llll}
f_0:=\frac{\partial}{\partial t} & f_1:=\frac{1}{f}e_7 &
f_2:=\frac{1}{a}e_1 & f_3:=\frac{1}{a}e_2\\ &&&\\
f_4:=\frac{1}{b}e_4 & f_5:=\frac{1}{b}e_3 & f_6:=\frac{1}{c}e_6 &
f_7:=\frac{1}{c}e_5\\
\end{array}
\end{equation}

The matrix representation of the isotropy action of $\mathfrak{u}(1)_{k,l}$ with respect to the basis (\ref{CBasisNkl}) can be
identified with a one-dimensional subgroup of Spin($7$). The $SU(3)$-invariant four-form $\Omega$ which is determined by
(\ref{CBasisNkl}) thus is well-defined. Since $(f_i)_{1\leq i\leq 7}$ is orthonormal with respect to (\ref{MetricsNkl}),
$(f_i)_{1\leq i\leq 7}$ defines a $G_2$-structure on $N^{k,l}$ whose associated metric is an arbitrary $SU(3)$-invariant
one. If we wrote down that $G_2$-structure explicitly, we would see that it contains odd powers of $a$, $b$, $c$, and $f$, as we have
remarked above. We calculate $d\Omega$ and see that $d\Omega=0$ is equivalent to

\begin{equation}
\label{HolRedNkl}
\begin{split}
\frac{a'}{a} &
=\frac{b^2+c^2-a^2}{abc}+\frac{-k-l}{2\Delta}\frac{f}{a^2}\\
\frac{b'}{b} &
=\frac{c^2+a^2-b^2}{abc}+\frac{l}{2\Delta}\frac{f}{b^2}\\
\frac{c'}{c} &
=\frac{a^2+b^2-c^2}{abc}+\frac{k}{2\Delta}\frac{f}{c^2}\\
\frac{f'}{f} & =-\frac{-k-l}{2\Delta}\frac{f}{a^2}
-\frac{l}{2\Delta}\frac{f}{b^2}-\frac{k}{2\Delta}\frac{f}{c^2}\\
\end{split}
\end{equation}

In the above system, $\Delta$ denotes $k^2+lk+l^2$. Furthermore, we have replaced $t$ by $-t$ for cosmetic reasons. This
convention will be maintained throughout the article. We remark that the system (\ref{HolRedNkl}) has also been deduced by
Kanno and Yasui \cite{KanI}. If $k=1$ and $l=0$, we can define $\Omega$ by (\ref{CBasisNkl}), too, and $d\Omega=0$ again is
equivalent to (\ref{HolRedNkl}). In that situation, we can choose the metric on the principal orbit as an arbitrary diagonal
one. If $k=l=1$, $f_2$ and $f_3$ have different coefficients, since we may have $g(e_1,e_1)\neq g(e_2,e_2)$. In that case
we choose the following basis $(f_i)_{0\leq i\leq 7}$ which yields an $SU(3)$-invariant Spin($7$)-structure whose
associated metric is an arbitrary diagonal one.

\begin{equation}
\label{CBasisN11}
\begin{array}{llll}
f_0:=\frac{\partial}{\partial t} & f_1:=\frac{1}{f}e_7 &
f_2:=\frac{1}{a_1}e_1 & f_3:=\frac{1}{a_2}e_2\\ &&&\\
f_4:=\frac{1}{b}e_4 & f_5:=\frac{1}{b}e_3 & f_6:=\frac{1}{c}e_6 &
f_7:=\frac{1}{c}e_5\\
\end{array}
\end{equation}

The equation $d\Omega=0$ is equivalent to the slightly more complicated system

\begin{equation}
\label{HolRedN11}
\begin{split}
\frac{a_1'}{a_1} & =
\frac{b^2+c^2-a_1^2}{a_1bc}+3\frac{a_1^2-a_2^2}{a_1a_2f}
-\frac{1}{3}\frac{f}{a_1a_2}\\ \frac{a_2'}{a_2} & =
\frac{b^2+c^2-a_2^2}{a_2bc}+3\frac{a_2^2-a_1^2}{a_1a_2f}
-\frac{1}{3}\frac{f}{a_1a_2}\\ \frac{b'}{b} & =
\frac{1}{2}\frac{a_1^2+c^2-b^2}{a_1bc}
+\frac{1}{2}\frac{a_2^2+c^2-b^2}{a_2bc}+\frac{1}{6}\frac{f}{b^2}\\
\frac{c'}{c} & = \frac{1}{2}\frac{a_1^2+b^2-c^2}{a_1bc}
+\frac{1}{2}\frac{a_2^2+b^2-c^2}{a_2bc}+\frac{1}{6}\frac{f}{c^2}\\
\frac{f'}{f} & =
-3\frac{(a_1-a_2)^2}{a_1a_2f}+\frac{1}{3}\frac{f}{a_1a_2}
-\frac{1}{6}\frac{f}{b^2}-\frac{1}{6}\frac{f}{c^2}\\
\end{split}
\end{equation}

In the above equations, we again have replaced $t$ by $-t$. The system (\ref{HolRedN11})
was also deduced by Kanno and Yasui \cite{KanII}. We want to know if there are
any $SU(3)$-invariant $G_2$-structures on the Aloff-Wallach spaces except those which
are determined by a basis of type (\ref{CBasisNkl}) or (\ref{CBasisN11}). Let $g$ be an arbitrary
homogeneous metric on $N^{k,l}$. In Lemma \ref{G2StrucSpace}, we have proven that the
space of all $SU(3)$-invariant $G_2$-structures on $N^{k,l}$ whose associated metric is
$g$ and whose orientation is fixed can be described by

\begin{equation}
\text{Norm}_{SO(7)}U(1)_{k,l}/\text{Norm}_{G_2}U(1)_{k,l}\:.
\end{equation}

In this situation, $U(1)_{k,l}$ is identified with its representation on $\mathfrak{m}$ or $\text{Im}(\mathbb{O})$
respectively. We first investigate the problem on the Lie algebra level. Let $G\in\{G_2,SO(7)\}$ and $\mathfrak{g}$ be the Lie
algebra of $G$. The tangent space of $\text{Norm}_{G}U(1)_{k,l}$ is

\begin{equation}
\text{Norm}_{\mathfrak{g}}\mathfrak{u}(1)_{k,l}:= \{ x\in\mathfrak{g} | [z,x]\in\mathfrak{u}(1)_{k,l}\quad\forall
z\in\mathfrak{u}(1)_{k,l}\}\:.
\end{equation}

The Lie algebra $\mathfrak{u}(1)_{k,l}$  is generated by a single $z\in\mathfrak{gl}(7,\mathbb{R})$. Let $\kappa$ be the
Killing form of $\mathfrak{g}$. $x$ is contained in $\text{Norm}_{\mathfrak{g}}\mathfrak{u}(1)_{k,l}$ if and only if

\begin{equation}
[z,x]=\lambda z\quad\text{for a}\quad \lambda\in\mathbb{R}\:.
\end{equation}

From this relation, it follows that

\begin{equation}
0 = \kappa(x,[z,z]) = \kappa([x,z],z) = -\lambda\:\kappa(z,z)\:.
\end{equation}

The above equation is satisfied only if $\lambda=0$ and thus we have

\begin{equation}
\text{Norm}_{\mathfrak{g}}\mathfrak{u}(1)_{k,l} =
\{x\in\mathfrak{g} | [z,x] = 0 \} =:
C_{\mathfrak{g}}\:\mathfrak{u}(1)_{k,l}\:.
\end{equation}

We are going to determine the centralizer $C_{\mathfrak{g}}\:\mathfrak{u}(1)_{k,l}$. First,
we work with the complexification of $C_{\mathfrak{g}}\:\mathfrak{u}(1)_{k,l}$, since this
will simplify some of our arguments. Any $x\in\mathfrak{g}\otimes\mathbb{C}$ has a
Cartan decomposition

\begin{equation}
x=h+\sum\limits_{\alpha\in\Phi}\mu_{\alpha}x_{\alpha}\quad\text{with}\:\mu_{\alpha}\in\mathbb{C}\:.
\end{equation}

In the above formula, $h$ is an element of a fixed Cartan subalgebra $\mathfrak{h}$ of
$\mathfrak{g}\otimes\mathbb{C}$. Furthermore, $\Phi$ is the root system of
$\mathfrak{g}\otimes\mathbb{C}$ and $x_{\alpha}$ is a suitable generator of the
root space $L_{\alpha}$ of $\alpha$. We assume without loss of generality that
$z\in\mathfrak{h}$. With this notation, the centralizer can be described as follows.

\begin{equation}
C_{\mathfrak{g}\otimes\mathbb{C}}\:(\mathfrak{u}(1)_{k,l}\otimes\mathbb{C})
= \left\{\begin{array}{l} \\ \\ \end{array}\!\!\!\!
x\in\mathfrak{g}\otimes\mathbb{C}\right.\left|[z,x]=
\sum\limits_{\alpha\in\Phi}\alpha(z)\mu_{\alpha}x_{\alpha}=0\right\}\:.
\end{equation}

Let $\Phi':=\{\alpha\in\Phi|\alpha(z)=0\}$. The above formula can be simplified to

\begin{equation}
C_{\mathfrak{g}\otimes\mathbb{C}}\:(\mathfrak{u}(1)_{k,l}\otimes\mathbb{C})=
\mathfrak{h}\oplus \bigoplus_{\alpha\in\Phi'} L_{\alpha}\:.
\end{equation}

We specialize to the case $\mathfrak{g}=\mathfrak{so}(7,\mathbb{C})$. Let $(\theta_1,\theta_2, \theta_3)$ be a basis of the
dual $\mathfrak{h}^{\ast}$ of $\mathfrak{h}$ such that

\begin{equation}
\Phi = \{\pm\theta_i|1\leq i\leq 3\}\cup\{\pm\theta_i\pm\theta_j|1\leq
i<j\leq 3\}\:.
\end{equation}

For reasons of simplicity we identify $\mathfrak{h}$ and $\mathfrak{h}^{\ast}$ by the Killing form. The action of $z$ on
$\mathfrak{m}$ is described by the weights (\ref{WeightsU1kl}). For a suitable choice of $z$ and
$(\theta_1,\theta_2,\theta_3)$ we thus have

\begin{equation}
z=(k-l)\theta_1+(2k+l)\theta_2+(k+2l)\theta_3\:.
\end{equation}

Let $\alpha=\sum_{i=1}^3 \alpha_i\theta_i\in\Phi$. The equation $\alpha(z)=0$ is equivalent to

\begin{equation}
(k-l)\alpha_1+(2k+l)\alpha_2+(k+2l)\alpha_3=0\:.
\end{equation}

$\theta_i(z)$ vanishes if and only if the $i^{th}$ of the
three coefficients $k-l,2k+l,k+2l$ equals zero. Analogously, the
root $\pm\theta_i \pm\theta_j$ is contained in $\Phi'$ if and only
if the $i^{th}$ and the $j^{th}$ coefficient coincide up to the
sign. We are now able to describe the normalizer in each of the
three cases.

\begin{itemize}
    \item Let $N^{k,l}$ be a generic Aloff-Wallach space. Since the set $\Phi'$ is empty,
    $\text{Norm}_{\mathfrak{so}(7,\mathbb{C})}(\mathfrak{u}(1)_{k,l} \otimes
    \mathbb{C}) = \mathfrak{h}$. The normalizer $\text{Norm}_{\mathfrak{so}(7)}\mathfrak{u}(1)_{k,l}$ therefore has to be
    isomorphic to $3\mathfrak{u}(1)$. More precisely, it has the following matrix representation with respect to the basis
    $(e_i')_{1\leq i\leq 7}:=(\tfrac{e_i}{\|e_i\|})_{1\leq i\leq 7}$ of $\mathfrak{m}$.

    \begin{equation}
    \label{CsaSO7}
    \left\{\left(\,\begin{array}{ccccccc}
    \hhline{--~~~~~}
    \multicolumn{1}{|c}{0} & \multicolumn{1}{c|}{a} &&&&&\\
    \multicolumn{1}{|c}{-a} & \multicolumn{1}{c|}{0} &&&&&\\
    \hhline{----~~~}
    && \multicolumn{1}{|c}{0} & \multicolumn{1}{c|}{b} &&&\\
    && \multicolumn{1}{|c}{-b} & \multicolumn{1}{c|}{0} &&&\\
    \hhline{~~----~}
    &&&& \multicolumn{1}{|c}{0} & \multicolumn{1}{c|}{c} & \\
    &&&& \multicolumn{1}{|c}{-c} & \multicolumn{1}{c|}{0} & \\
    \hhline{~~~~---}
    &&&&&& \multicolumn{1}{|c|}{0} \\
    \hhline{~~~~~~-}
    \end{array}\,\right)\right.\left|
    \begin{array}{c} \\ \\ \\ \\ \\ \\ \\ \\ \end{array}
    a,b,c\in\mathbb{R}\right\}\:.
    \end{equation}

    \item If $k=1$ and $l=0$, the set $\Phi'$ consists of the two roots $\pm (\theta_1 - \theta_3)$ and
    thus is a root system of type $A_1$. Since
    $\text{Norm}_{\mathfrak{so}(7)} \mathfrak{u}(1)_{1,0}$ is the compact real form of
    $\text{Norm}_{\mathfrak{so}(7,\mathbb{C})}(\mathfrak{u}(1)_{1,0} \otimes\mathbb{C})$, it is isomorphic to
    $\mathfrak{su}(2)\oplus 2\mathfrak{u}(1)$. The semisimple part of the normalizer acts irreducibly on $V_1\oplus V_3$ such
    that $(e_1',e_2',e_5',e_6')$ is identified with the standard basis of $\mathbb{C}^2$.

    \item If $k=l=1$, $\Phi'$ consists of the two pairs $\pm\theta_1$ and $\pm (\theta_2 - \theta_3)$. Since $A_1\times A_1$ is the
    Dynkin diagram of $\mathfrak{sl}(2,\mathbb{C}) \oplus\mathfrak{sl}(2,\mathbb{C})$, the real Lie algebra
    $\text{Norm}_{\mathfrak{so}(7)}\mathfrak{u}(1)_{1,1}$ has to be isomorphic to $2\mathfrak{su}(2) \oplus
    \mathfrak{u}(1)$. One of the two simple summands acts by its three-dimensional representation on
    $\text{span}(e_1',e_2',e_7')$. The other one acts irreducibly on $V_2\oplus V_3$ such that we can identify
    $(e_3',e_4',e_5',e_6')$ with the standard basis of $\mathbb{C}^2$.
\end{itemize}

By intersecting $\text{Norm}_{\mathfrak{so}(7)}\mathfrak{u}(1)_{k,l}$ with $\mathfrak{g}_2$ we are able to determine
$\text{Norm}_{\mathfrak{g}_2}\mathfrak{u}(1)_{k,l}$. Again, we consider the three cases separately.

\begin{itemize}
    \item Let $N^{k,l}$ be a generic Aloff-Wallach space. $\text{Norm}_{\mathfrak{g}_2}\mathfrak{u}(1)_{k,l}$
    is isomorphic to the two-dimensional subalgebra of (\ref{CsaSO7}) which is a Cartan subalgebra of
    $\mathfrak{g}_2$.

    \item Let $k=1$ and $l=0$. By an explicit calculation, we see that the Lie algebra action of the semisimple part of
    $\text{Norm}_{\mathfrak{so}(7)}\mathfrak{u}(1)_{1,0}$ on $\omega\in \bigwedge^3 \text{Im}(\mathbb{O})^{\ast}$ is trivial.
    The semisimple part therefore is a subalgebra of $\mathfrak{g}_2$. Since $\mathfrak{g}_2$ is of rank
    two, $\text{Norm}_{\mathfrak{g}_2} \mathfrak{u}(1)_{1,0}$ is isomorphic to $\mathfrak{su}(2) \oplus \mathfrak{u}(1)$.

    \item Finally, let $k=l=1$. Since $\text{Norm}_{\mathfrak{g}_2}\mathfrak{u}(1)_{1,1}$ contains the Cartan subalgebra of
    $\mathfrak{g}_2$, it is of rank two. There are four subalgebras of $2\mathfrak{su}(2)\oplus \mathfrak{u}(1)$ which have
    rank two. One of them is $2\mathfrak{su}(2)$. The other three are isomorphic to $\mathfrak{su}(2) \oplus \mathfrak{u}(1)$
    where the semisimple part is either an ideal of $2\mathfrak{su}(2)$ or diagonally embedded. By similar techniques as in the
    previous case, we see that the semisimple part is diagonally embedded. More precisely, it acts irreducibly on
    $\text{span}(e_1',e_2',e_7')$ and $\text{span}(e_3',e_4',e_5',e_6')$.
\end{itemize}

We are now able to describe both normalizers of the Lie group $U(1)_{k,l}$. For
the following considerations, we denote the identity component of
a Lie group $G$ by $G_e$. $(\text{Norm}_{SO(7)} U(1)_{k,l})_e$
acts transitively and almost freely on any connected component of
$\text{Norm}_{SO(7)} U(1)_{k,l} / \text{Norm}_{G_2} U(1)_{k,l}$.
For our purpose, it therefore suffices to describe the quotient
$(\text{Norm}_{SO(7)} U(1)_{k,l})_e /$ $(\text{Norm}_{G_2}
U(1)_{k,l})_e$, which covers any connected component of\linebreak
$\text{Norm}_{SO(7)} U(1)_{k,l} / \text{Norm}_{G_2} U(1)_{k,l}$.

\begin{itemize}
    \item If $N^{k,l}$ is generic, we have

    \begin{equation}
    (\text{Norm}_{SO(7)}U(1)_{k,l})_e/(\text{Norm}_{G_2}U(1)_{k,l})_e\cong U(1)^3/U(1)^2\cong S^1\:.
    \end{equation}

    $\!\!\!\!\!\!\!$Let $T$ be a one-dimensional connected Lie subgroup of the maximal torus of $SO(7)$ whose Lie algebra is
    (\ref{CsaSO7}). We choose $T$ in such a way that $G_2\cap T$ is discrete. The action of $T$ on a fixed homogeneous
    $G_2$-structure $\omega$ generates a connected component of the space of all $SU(3)$-invariant $G_2$-structures which have
    the same associated metric and orientation as $\omega$.

    \item If $k=1$ and $l=0$, we have

    \begin{equation}
    \begin{array}{rcl}
    (\text{Norm}_{SO(7)}U(1)_{k,l})_e/(\text{Norm}_{G_2}U(1)_{k,l})_e & \cong & (U(2)\times U(1))/U(2) \\
    & \cong & S^1\:.\\
    \end{array}
    \end{equation}

    $\!\!\!\!\!\!\!$The space of all $SU(3)$-invariant $G_2$-structures with a fixed associated metric and orientation
    can therefore be described as in the previous case.

    \item The semisimple part of $\text{Norm}_{\mathfrak{g}_2}\mathfrak{u}(1)_{1,1}$ is diagonally embedded into the
    semisimple part $2\mathfrak{su}(2)$ of $\text{Norm}_{\mathfrak{so}(7)}\mathfrak{u}(1)_{1,1}$. We denote the ideal of
    $2\mathfrak{su}(2)$ which acts non-trivially on $\text{span}(e_1',e_2',e_7')$ by $\mathfrak{su}(2)'$. The abelian part of
    $\text{Norm}_{\mathfrak{g}_2}\mathfrak{u}(1)_{1,1}$ is not a subgroup of $\mathfrak{su}(2)'$. We can conclude that
    the Lie group which corresponds to $\mathfrak{su}(2)'$ is isomorphic to $SO(3)$ and acts transitively and freely
    on $(\text{Norm}_{SO(7)}$ $U(1)_{k,l})_e/$ $(\text{Norm}_{G_2}U(1)_{k,l})_e$. The set of all $SU(3)$-invariant
    $G_2$-structures on $N^{1,1}$ whose associated metric and orientation is
    fixed can therefore be generated by an $SO(3)$-action on a single $G_2$-structure.
\end{itemize}

Until now, we have only determined the type of the connected components of
$\text{Norm}_{SO(7)}U(1)_{k,l}/\text{Norm}_{G_2}U(1)_{k,l}$, but
not their number. We therefore cannot rule out that there are
further homogeneous $G_2$-structures which cannot be obtained by
the above $U(1)$- or $SO(3)$-actions. It may be possible that the
additional $G_2$-structures could be extended to new examples of
parallel cohomogeneity-one $\text{Spin}(7)$-structures. We will
nevertheless not discuss this issue further. There are three
reasons for our decision. First, the study of the
$\text{Spin}(7)$-structures whose restriction to a principal orbit
is one of the $G_2$-structures which we have constructed so far is
already a rewarding project even if further examples existed.
Second, it may be possible that we switch to another connected
component of $\text{Norm}_{SO(7)}U(1)_{k,l}/\text{Norm}_{G_2}U(1)_{k,l}$ if we
change the sign of some of the functions $a$, $a_1$, $a_2$, $b$,
$c$, and $f$. Therefore, we may already describe more than one or
even all connected components by our ansatz for $(f_i)_{1\leq i\leq 7}$.
A third reason is that the results which we have found are sufficient to
make statements on the holonomy of our metrics.

We start with the generic Aloff-Wallach spaces $N^{k,l}$. Let $g$
be a fixed metric on $N^{k,l}$ and $\omega$ be a $G_2$-structure
whose associated metric is $g$. By an explicit calculation we see
that $\omega$ is cocalibrated only if it is induced by the basis
$(f_i)_{1\leq i\leq 7}$ of $\mathfrak{m}$ which we have defined on
page \pageref{CBasisNkl}. Let $\widetilde{g}$ be a
cohomogeneity-one metric whose holonomy is contained in Spin($7$).
We assume that there is a principal orbit such that the
restriction of $\widetilde{g}$ to that orbit is $g$. Since the
holonomy bundle is $SU(3)$-invariant, any parallel
Spin($7$)-structure $\Omega$ whose associated metric is
$\widetilde{g}$ is $SU(3)$-invariant, too. Any parallel
Spin($7$)-structure induces a cocalibrated $G_2$-structure on the
principal orbit. Since the set of all $SU(3)$-invariant cocalibrated
$G_2$-structures is discrete, the set of all invariant parallel
Spin($7$)-structures is discrete, too. According to Lemma
\ref{SU4Lemma}, the holonomy of $\widetilde{g}$ cannot be
$SU(4)$ or one of its subgroups.

If the holonomy was $G_2$, there would exist a parallel vector
field $X$ on the manifold. Moreover, the space of all parallel
vector fields would be one-dimensional. Since that space is
$SU(3)$-invariant, $X$ is invariant, too, and of type $c_1(t)
\cdot \tfrac{\partial}{\partial t} + \tfrac{c_2(t)}{f(t)}\cdot
e_7$. The dual $c_1(t)\cdot dt + c_2(t)\cdot f(t)\cdot e^7$ of $X$
has to be a closed one-form. By calculating the exterior derivative,
we see that $c_2$ has to vanish for all $t$. Since the length of
$X$ has to be constant, $c_1$ has to be constant, too. If
$\tfrac{\partial}{\partial t}$ was parallel, we would have
$a'=b'=c'=f'=0$. In that situation, the manifold would be the
product of $N^{k,l}$ with a parallel $SU(3)$-invariant
$G_2$-structure and an interval. This is impossible since
$N^{k,l}$ is not a torus and we thus have proven that the holonomy
is exactly Spin($7$).

If the principal orbit is $N^{1,0}$, it follows by the same
arguments that the holonomy is all of Spin($7$). In the case where
$k=l=1$, the situation is more complicated, since the space of all
$G_2$-structures with a fixed associated metric is
three-dimensional. Let $\omega$ be the $G_2$-structure on
$N^{1,1}$ which is induced by the basis $(f_i)_{1\leq i\leq 7}$
from page \pageref{CBasisN11}. If the holonomy is contained in
$SU(4)$, there exists a map

\begin{equation}
\widetilde{\omega}:[0,\epsilon) \rightarrow {\bigwedge}^3 \mathfrak{m}^{\ast}
\end{equation}

such that $\widetilde{\omega}(0)=\omega$, $d\ast\widetilde{\omega}(s)=0$
for all $s\in[0,\epsilon)$, and the metric which is associated to
$\widetilde{\omega}(s)$ is the same as of $\omega$. We have proven
that $\widetilde{\omega}(s)$ can be obtained by the action of an
$A(s)\in SO(3)$ on $(e_1',e_2',e_7')$. We consider the case where
$A(s)$ is of type

\begin{equation}
\label{Avont}
\left(\,
\begin{array}{ccc}
\cos{s} & -\sin{s} & 0 \\
\sin{s} & \cos{s} & 0 \\
0 & 0 & 1 \\
\end{array}
\,\right)
\end{equation}

With the help of a short MAPLE program, we can show that if
$s\notin \pi\mathbb{Z}$, $d\ast\widetilde{\omega}(s)=0$ is
equivalent to $a_1(t) + a_2(t) = 0$. Under this assumption, it
follows from (\ref{HolRedN11}) that $a_1(t)^2 = b(t)^2 + c(t)^2$.
Under these assumptions (\ref{HolRedN11}) is explicitly solvable
and we obtain the metrics of Bazaikin and Malkovich \cite{Baz2}.
In \cite{Baz2} it is proven that the holonomy is all of $SU(4)$,
except in a limiting case where we obtain Calabis \cite{Calabi}
hyperk\"ahler metric on $T^{\ast} \mathbb{CP}^2$.

If we replace the $A$ from (\ref{Avont}) by another one-parameter subgroup
of $SO(3)$, we obtain more complicated conditions on the
Spin($7$)-structures, which may or may not be satisfiable. The question if there
are further metrics whose holonomy is a proper subgroup of Spin($7$) is
a subject of future research.

At the end of this section, we classify all possible singular orbits of a
cohomogeneity-one manifold whose principal orbit is an Aloff-Wallach
space. This classification can also be found in Gambioli \cite{Gambioli}.
Our results are summarized by the following lemma.

\begin{Le} \label{NklSing}
Let $U(1)$ be embedded into $SU(3)$ as $U(1)_{k,l}$. Furthermore, let $K$ be a connected, closed group with
$U(1)_{k,l}\subsetneq K\subseteq SU(3)$. The Lie algebra of $K$ we denote by $\mathfrak{k}$. In this situation, $\mathfrak{k}$
and $K$ can be found in the table below. Moreover, $K/U(1)_{k,l}$ and $SU(3)/K$ satisfy the following topological conditions.

\begin{center}
\begin{small}
\begin{tabular}{|l|l|l|l|l|}
\hline

$\mathfrak{k}$ & $K$ & $K/U(1)_{k,l}$ & $SU(3)/K$ & Condition on
$k$ and $l$\\

\hline \hline

$2\mathfrak{u}(1)$ & $U(1)^2$ & $\cong S^1$ & $= SU(3)/U(1)^2$ &
\\

\hline

$\mathfrak{su}(2)$ & $SU(2)$ & $\cong S^2$ & $\cong S^5$ & $k\cdot
l\cdot (-k-l)=0$\\

\hline

$\mathfrak{su}(2)$ & $SO(3)$ & $\cong S^2$ & $=SU(3)/SO(3)$ &
$k\cdot l\cdot (-k-l)=0$\\

\hline

$\mathfrak{su}(2)\oplus\mathfrak{u}(1)$ & $U(2)$ & $\cong
S^3/\mathbb{Z}_{|k+l|}$ & $\cong\mathbb{CP}^2$ & $(k,l)
\neq (1,-1)$\\

\hline

$\mathfrak{su}(2)\oplus\mathfrak{u}(1)$ & $U(2)$ & $\cong
S^2\times S^1$ & $\cong\mathbb{CP}^2$ & $(k,l)=(1,-1)$\\

\hline

$\mathfrak{su}(3)$ & $SU(3)$ & $=N^{k,l}\not\cong S^7/\Gamma$ & &
\\

\hline
\end{tabular}
\end{small}
\end{center}

In the above table, $\Gamma$ denotes an arbitrary discrete
subgroup of $O(8)$ and the group $\mathbb{Z}_{|k+l|}$, by which we
divide $S^3$, is explicitly described by (\ref{Zkl}).
\end{Le}

\begin{proof} The fact that $\mathfrak{k}$ has to be an $\mathfrak{u}(1)_{k,l}$-module, reduces the number of subspaces of
$\mathfrak{su}(3)$ which we have to consider. We have to check for all $\mathfrak{u}(1)_{k,l}$-modules $\mathfrak{k}$ with
$\mathfrak{u}(1)_{k,l} \subsetneq \mathfrak{k}\subseteq\mathfrak{su}(3)$ if they are closed under the Lie bracket and if
$K/U(1)_{k,l}$ is covered by a sphere. If $N^{k,l}$ is an exceptional Aloff-Wallach space, $\mathfrak{m}$ contains pairs
$(U,U')$ of equivalent submodules. In that situation, there are infinitely many submodules of $\mathfrak{k}$ which are
transversely embedded into $U\oplus U'$. Therefore, we often have to distinguish between the different types of Aloff-Wallach
spaces in the course of this proof. We consider each of the possible values of $\dim{\mathfrak{k}}$ separately.

\underline{$\mathfrak{k}=\mathfrak{u}(1)_{k,l}\oplus W$, $\dim{W}=1$:} If $(k,l)\neq(1,1)$, $V_4$ is the only one-dimensional
submodule of $\mathfrak{m}$ and the statements of the lemma are satisfied. If $k=l=1$, $V_1$ is trivial. We can choose the
generator of $W$ as an arbitrary $x\in (V_1 \oplus V_4)\setminus\{0\}$ and obtain the same statements on the singular
orbit.

\label{WNkl} \underline{$\mathfrak{k}=\mathfrak{u}(1)_{k,l}\oplus W$, $\dim{W}=2$:} If $N^{k,l}$ is generic, the only
two-dimensional submodules of $\mathfrak{m}$ are $V_1$, $V_2$, and $V_3$. In any case $\mathfrak{k}$ is not closed under the
Lie bracket.

Next, we assume that $k=1$ and $l=0$. $\mathfrak{u}(1)_{1,0}\oplus V_2$ is a Lie algebra, which is isomorphic to
$\mathfrak{su}(2)$ and we have $K=SU(2)$ and $K/U(1)_{1,0} \cong S^2$. Since the modules $V_1$ and $V_3$ are
$\mathfrak{u}(1)_{1,0}$-equivariantly isomorphic, there may exist a two-dimensional space $W$ which is transversely embedded
into $V_1\oplus V_3$ such that $\mathfrak{k}$ is closed under the Lie bracket. It turns out that the possible $\mathfrak{k}$
are precisely

\begin{equation}
\mathfrak{k}=\text{span}(e_8,\:\tfrac{1}{2}\sqrt{2}\:e_1 + \tfrac{1}{2}
\sqrt{2}\:e_5,\:-\tfrac{1}{2}\sqrt{2}\:e_2 - \tfrac{1}{2}\sqrt{2}\:e_6)
\end{equation}

and its conjugates with respect to $U(1)_{1,1}$. Since $\mathfrak{k}$ is conjugate to the standard embedding of
$\mathfrak{so}(3)$ into $\mathfrak{su}(3)$, we have $K\cong SO(3)$, $K/U(1)_{1,-1}\cong S^2$ and the singular
orbit is the symmetric space $SU(3)/SO(3)$.

If $k=l=1$, $\mathfrak{k}$ again cannot be a Lie algebra.

\underline{$\mathfrak{k}=\mathfrak{u}(1)_{k,l}\oplus W$, $\dim{W}=3$:} Let $N^{k,l}$ be a generic Aloff-Wallach space. In this
situation, we have $W=V_i\oplus V_4$ for an $i\in\{1,2,3\}$. As a Lie algebra $\mathfrak{k}$ is isomorphic to
$\mathfrak{u}(2)$. We assume without loss of generality that $i=1$. The Lie group $K$ is given by
$S(U(2)\times U(1))$.

We analyze the topology of $K/U(1)_{k,l}$. Let $\pi:SU(2)\rightarrow K/U(1)_{k,l}$ be the map with $\pi(h) := hU(1)_{k,l}$,
where $SU(2)$ is embedded into $SU(3)$ such that its Lie algebra is $\mathfrak{u}(1)_{1,-1}\oplus V_1$. It is easy to see that $\pi$
is a covering map. Its kernel is $SU(2)\cap U(1)_{k,l}$. Since $k$ and $l$ are coprime, this intersection is, except for
$(k,l)=(1,-1)$,

\begin{equation}
\label{Zkl}
\left\{\left(\,\begin{array}{cc}
e^{2\pi i\tfrac{m}{k+l}} & 0 \\
0 & e^{-2\pi i\tfrac{m}{k+l}}\\
\end{array}\,\right)\right.
\left|\begin{array}{l} \\ \\ \\ \end{array}
m\in\mathbb{Z}\right\}\:.
\end{equation}

The quotient $K/U(1)_{k,l}$ thus is the lens space $L(k+l,1)$. For reasons of brevity, we will denote $K/U(1)_{k,l}$ by
$S^3/\mathbb{Z}_{|k+l|}$, since $\mathbb{Z}_{|k+l|}$ will always be the above discrete group.  The
quotient $SU(3)/K$ is diffeomorphic to $\mathbb{CP}^2$.

Next, we consider the exceptional case $k=1$, $l=0$. If we choose $W$ as $V_i\oplus V_4$ with $i\in\{1,3\}$, $K/U(1)_{1,0}$ is
the sphere $S^3$. For $W=V_2\oplus V_4$, we have $K/U(1)_{1,0} = S^2\times S^1$. In order to finish the first exceptional
case, we have to check if we can choose $W=W'\oplus V_4$ with $W'\subseteq V_1\oplus V_3$ but $W'\notin\{V_1,V_3\}$. Since the
Cartan subalgebra $\text{span}(e_7,e_8)$, which we shortly denote by $2\mathfrak{u}(1)$, is contained in $\mathfrak{k}$, $W'$
has to be a $2\mathfrak{u}(1)$-module. The $q$-orthogonal complement of $2\mathfrak{u}(1)$ in $\mathfrak{su}(3)$ decomposes
with respect to $2\mathfrak{u}(1)$ into $V_1\oplus V_2\oplus V_3$ and the three modules are pairwise inequivalent. Therefore,
the case $W'\notin\{V_1,V_3\}$ cannot happen.

We finally consider the case $k=l=1$. $\mathfrak{k}$ has to be isomorphic to the Lie algebra $\mathfrak{su}(2)
\oplus \mathfrak{u}(1)$. As above, we decompose $\mathfrak{k}$ into the direct sum $\mathfrak{h} \oplus W'$ of
$\mathfrak{u}(1)_{1,1}$-modules, where $\mathfrak{h}$ is a Cartan subalgebra of $\mathfrak{k}$ which contains
$\mathfrak{u}(1)_{1,1}$ and $W'$ is its $q$-orthogonal complement. Since $\text{span}(e_1,e_2,e_7)$ is the centralizer of
$\mathfrak{u}(1)_{1,1}$, $\mathfrak{h}$ can be any of the following spaces.

\begin{equation}
\mathfrak{h}=\mathfrak{u}(1)_{1,1}\oplus\text{span}(x)\quad\text{with}\:\:\:
x\in\text{span}(e_1,e_2,e_7)\setminus\{0\}
\end{equation}

$W'$ is either a submodule of $\text{span}(e_1,e_2,e_7)$ or of $V_2\oplus V_3$. In the first case, it has to be the complement
of $\text{span}(x)$ in $\text{span}(e_1,e_2,e_7)$ and $\mathfrak{k}$ therefore is $2\mathfrak{u}(1)\oplus V_1$, which yields
$K/U(1)_{1,1} = S^3/\mathbb{Z}_2$ and $SU(3)/K = \mathbb{CP}^2$.

We turn to the case where $W'\subseteq V_2\oplus V_3$. There exists an

\begin{equation}
A\in\left\{\left(\,
\begin{array}{cc}
\hhline{-~}
\multicolumn{1}{|c|}{A'} & \\
\hhline{--}
& \multicolumn{1}{|c|}{1} \\
\hhline{~-}
\end{array}\,\right)\right|\left.
\begin{array}{l} \\ \\ \end{array}\!\!\!\!
A'\in SU(2)\right\}\quad
\text{with}\:\:\:
A\:\mathfrak{h}\:A^{-1}= 2\mathfrak{u}(1)\:.
\end{equation}

We conjugate our decomposition of $\mathfrak{k}$ by $A$ and obtain

\begin{equation}
A\:\mathfrak{k}\:A^{-1}=2\mathfrak{u}(1)\oplus AW'A^{-1}\:.
\end{equation}

Since $AW'A^{-1}$ is a $2\mathfrak{u}(1)$-module, we can prove by the same arguments as in the case $(k,l)=(1,0)$ that
$AW'A^{-1}$ is either $V_2$ or $V_3$. In both cases we obtain $K/U(1)_{1,1}=S^3$ and $SU(3)/K = \mathbb{CP}^2$.

\underline{$\mathfrak{k}=\mathfrak{u}(1)_{k,l}\oplus W$, $\dim{W}\in \{4,6\}$:} Since for any root $\alpha$ of $\mathfrak{k}$,
$-\alpha$ is a root, too, we have $\dim{\mathfrak{k}}\equiv \text{rank}\:\mathfrak{k}\:\:\:(\text{mod}\: 2)$. Therefore, the
rank of $\mathfrak{k}$ has to be odd. Since $\text{rank}\:\mathfrak{su}(3)=2$, the only possibility for
$\text{rank}\:\mathfrak{k}$ is in fact $1$. Since the only Lie algebras of rank $1$ which belong to a compact Lie group
are $\mathfrak{u}(1)$ and $\mathfrak{su}(2)$, we can exclude these cases.

\underline{$\mathfrak{k}=\mathfrak{u}(1)_{k,l}\oplus W$, $\dim{W}=5$:} The only six-dimensional Lie algebra of rank $\leq 2$
which belongs to a compact Lie group is $\mathfrak{su}(2)\oplus\mathfrak{su}(2)$. There is no Lie subalgebra of
$\mathfrak{su}(3)$ of this type and we therefore can exclude this case.

\underline{$\mathfrak{k}=\mathfrak{u}(1)_{k,l}\oplus W$, $\dim{W}=7$:} In this case, $\mathfrak{k}$ is all of
$\mathfrak{su}(3)$. Since $H^4(N^{k,l},\mathbb{Z}) = \mathbb{Z}_{k^2 + kl + l^2}$, $N^{k,l}$ is not a quotient of
$S^7$ by a discrete subgroup $\Gamma\subseteq O(8)$.
\end{proof}

\begin{Conv}
\label{NklConv2} We want to treat the three Lie algebras $2\mathfrak{u}(1)\oplus V_i$ with $i\in\{1,2,3\}$ in a uniform
manner. Therefore, we will drop the convention $k\geq l\geq 0$ whenever we consider the case $\mathfrak{k}\cong
\mathfrak{u}(2)$. This will allow us to fix $\mathfrak{k}$ as $2\mathfrak{u}(1)\oplus V_1$. Since we can replace $(k,l)$ by
$(-k,-l)$, we can still assume that $k\geq l$. We have implicitly used this convention in the statement of the lemma when
we described the topology of $K/U(1)_{k,l}$ as $S^3/\mathbb{Z}_{|k + l|}$ instead of $S^3/\mathbb{Z}_{|k|}$ or
$S^3/\mathbb{Z}_{|l|}$, which would be the case if $i\in\{2,3\}$.
\end{Conv}

The results of this section can be summarized as follows.

\begin{Th} \label{NklThm1} Let $(M,\Omega)$ be a $\text{Spin}(7)$-orbifold with a cohomogeneity-one action of $SU(3)$ which
preserves $\Omega$ and whose principal orbit is an Aloff-Wallach space. The metric associated to $\Omega$ we denote by $g$. In
this situation, the following statements are true.

\begin{enumerate}
    \item
    \begin{enumerate}
        \item If $N^{k,l}$ is generic, the matrix representation of $g$ with respect to the basis $(e_1,\ldots,e_7)$ from page
        \pageref{qbasisAW} is of type (\ref{MetricsNkl}).
        \item If $k=1$ and $l=0$, the matrix representation of $g$ with respect to the basis $(e_1,e_2,e_5,e_6,e_3,e_4,e_7)$
        is of type (\ref{MetricsN10}).
        \item If $k=1$ and $l=1$, the matrix representation of $g$ with respect to the basis $(e_1,e_2,e_7,e_3,e_4,e_5,e_6)$
        is of type (\ref{MetricsN11}).
    \end{enumerate}
    \item We assume that $N^{k,l}$ is generic or $N^{1,0}$ and that $g$ is diagonal with respect to the basis from page
    \pageref{qbasisAW}. Let $\mathcal{S}'$ be a connected component of the space $\mathcal{S}$ of all $SU(3)$-invariant
    Spin($7$)-structures on $M$ with the same associated metric as $\Omega$. We assume that $\mathcal{S}'$ contains
    a four-form $\widetilde{\Omega}$ which is obtained from a basis of type (\ref{CBasisNkl}). In this situation, $\Omega$
    can be parallel only if it coincides with $\widetilde{\Omega}$. In other words, (\ref{CBasisNkl}) is up to other connected
    components of $\mathcal{S}$ the only possible basis for a diagonal metric with holonomy contained in Spin($7$).
    \item
    \begin{enumerate}
        \item If $N^{k,l}$ is generic or $N^{1,0}$ and $\Omega$ is determined by a basis of type (\ref{CBasisNkl}), $\Omega$
        is parallel if and only if the system (\ref{HolRedNkl}) is satisfied. In that case, the holonomy of $g$ is all of
        $\text{Spin}(7)$.
        \item If $k=l=1$ and $\Omega$ is determined by a basis of type (\ref{CBasisN11}), $\Omega$ is parallel if and only if
        the system (\ref{HolRedN11}) is satisfied. In this case, the holonomy of $g$ is a subgroup of Spin($7$). If we also
        have $a_1(t)+a_2(t)=0$ for all $t$, the holonomy of $g$ is $SU(4)$ except in the case where we obtain the Calabi
        metric which has holonomy $Sp(2)$.
    \end{enumerate}
\end{enumerate}

If $M$ has a singular orbit, which has to be the case if $(M,\Omega)$ is parallel and complete, it can be found in the
table of Lemma \ref{NklSing}. For any choice of the principal and the singular orbit, the orbifold $M$ is a manifold,
except in the case where the singular orbit is $\mathbb{CP}^2$ and $|k+l|\neq 1$. In that case, $M$ is a
$D^4/\mathbb{Z}_{|k+l|}$-bundle over $\mathbb{CP}^2$.
\end{Th}

In the following three sections, we search for cohomogeneity-one metrics with special holonomy on the spaces of
the above theorem. We have to treat each possible combination of a principal and a singular orbit as a separate case,
since the equations for the holonomy reduction or their initial conditions will change.

\section{$SU(3)/U(1)^2$ as singular orbit}

\subsection{The generic Aloff-Wallach spaces as principal orbit}

Throughout this and the following sections we will assume that the
singular orbit is at $t=0$. In the situation of this subsection,
the Lie algebra of the isotropy group at the singular orbit is
given by $\mathfrak{u}(1)_{k,l} \oplus V_4$. We therefore have
$f(0)=0$ and $a(0),b(0),c(0)\neq 0$ as initial conditions. We take
a look at the equation

\begin{equation}
f' = -\frac{-k-l}{2\Delta}\frac{f^2}{a^2}
-\frac{l}{2\Delta}\frac{f^2}{b^2}-\frac{k}{2\Delta}\frac{f^2}{c^2}
\end{equation}

from the system (\ref{HolRedNkl}). We see that $f'(0)=0$ and by a complete induction we can prove that all higher derivatives
of $f$ at $t=0$ vanish, too. Since the metric has to be analytic, it follows that $f$ vanishes identically. If we insert
$f\equiv 0$ into (\ref{HolRedNkl}), we obtain the equations for a cohomogeneity-one metric with principal orbit $SU(3)/U(1)^2$
and holonomy $G_2$. Since those equations were investigated by Cleyton and Swann \cite{Cley}, we will not discuss this
issue further.

Our next aim is to classify all cohomogeneity-one Einstein metrics
with our fixed orbit structure. In order to apply Theorem
\ref{EinsteinCohom1Solution}, we have to decompose certain
$2\mathfrak{u}(1)$-modules. The fact that we need these
decompositions in the later sections anyway is a further
motivation for our task. As in Section \ref{CohomOneSection}, we
denote the tangent space of the singular orbit by $\mathfrak{p}$
and its normal space by $\mathfrak{p}^{\perp}$. We need to
describe the spaces $W_m^h$ and $W_2^v$ of all
$2\mathfrak{u}(1)$-equivariant maps from
$S^m(\mathfrak{p}^{\perp})$ and $S^2(\mathfrak{p}^{\perp})$ into
$S^2(\mathfrak{p})$ and $S^2(\mathfrak{p})^{\perp}$. Therefore, we
decompose $S^2(\mathfrak{p})$ into irreducible
$2\mathfrak{u}(1)$-submodules. $\mathfrak{p}$ is the direct sum of
the modules $V_1$, $V_2$, and $V_3$, whose weights are given by

\begin{equation}
V_1 = \mathbbm{V}_{3k+3l,k-l}\:, \quad V_2 = \mathbbm{V}_{3l,2k+l}\:, \quad V_3 =
\mathbbm{V}_{-3k,k+2l}\:.
\end{equation}

In the above formula $\mathbbm{V}_{r,s}$ denotes the two-dimensional real $2\mathfrak{u}(1)$-module on which $e_7$ acts with
weight $r$ and $e_8$ acts with weight $s$. With the help of the relations

\begin{equation}
\begin{array}{rcl}
S^2(\mathbbm{V}_{r,s}) & = & \mathbbm{V}_{2r,2s} \oplus \mathbb{R}
\\ \mathbbm{V}_{r_1,s_1}\otimes \mathbbm{V}_{r_2,s_2} & = &
\mathbbm{V}_{r_1+r_2, s_1+s_2} \oplus
\mathbbm{V}_{r_1-r_2,s_1-s_2} \\
\mathbbm{V}_{r,s} & = & \overline{\mathbbm{V}_{-r,-s}} \cong \mathbbm{V}_{-r,-s}
\end{array}
\end{equation}

we see that

\begin{equation}
\begin{aligned}
S^2(\mathfrak{p}) & = \mathbbm{V}_{6k+6l,2k-2l} \oplus
\mathbbm{V}_{6l,4k+2l} \oplus \mathbbm{V}_{-6k,2k+4l} \\
& \quad \oplus \mathbbm{V}_{3k+6l,3k} \oplus
\mathbbm{V}_{3k,-k-2l} \oplus \mathbbm{V}_{3l,2k+l} \oplus
\mathbbm{V}_{6k+3l,-3l} \\ & \quad \oplus
\mathbbm{V}_{-3k+3l,3k+3l} \oplus \mathbbm{V}_{3k+3l,k-l}
\oplus 3\mathbb{R}\:.\\
\end{aligned}
\end{equation}

Next, we have to describe the action of $2\mathfrak{u}(1)$ on
$\mathfrak{p}^{\perp}$. The tangent vector
$\tfrac{\partial}{\partial t}\in\mathfrak{p}^{\perp}$ is fixed by
the action of $U(1)_{k,l}$. $e_8$ therefore acts trivially on
$\mathfrak{p}^{\perp}$. $e_7\in 2\mathfrak{u}(1)$ generates the
group $U(1)_{2l+k,-2k-l} \subseteq SU(3)$, which we will shortly
denote by $U(1)'$. The circle $U(1)^2/U(1)_{k,l}$, where $U(1)^2$
as usual denotes the subgroup of all diagonal matrices in $SU(3)$,
can be considered as a subset of $\mathfrak{p}^{\perp}$. Moreover,
the orbit of the $U(1)'$-action on a point $p\in
\mathfrak{p}^{\perp}\setminus\{0\}$ can be identified with a loop
in the space $U(1)^2/U(1)_{k,l}$. The weight of the $U(1)'$-action
on $\mathfrak{p}^{\perp}$ coincides with the homotopy class of
that loop, which is an integer. That number is the same as the
number of all $t\in [0,2\pi)$ such that $\exp{(t\cdot e_7)} \in
U(1)_{k,l}$. After a short calculation we can conclude that

\begin{equation}
\mathfrak{p}^\perp = \mathbbm{V}_{2(k^2+lk+l^2),0} \:.
\end{equation}

We are now able to decompose $S^m(\mathfrak{p}^{\perp})$ into irreducible $2\mathfrak{u}(1)$-submodules and obtain

\begin{equation}
S^m(\mathfrak{p}^\perp) = \left\{
\begin{array}{lll}
\mathbbm{V}_{2m(k^2+lk+l^2),0} & \oplus
\mathbbm{V}_{2(m-2)(k^2+lk+l^2),0} & \\
& \oplus \ldots\oplus
\mathbbm{V}_{4(k^2+lk+l^2),0} \oplus \mathbb{R} & \text{if $m$ is
even} \\ \mathbbm{V}_{2m(k^2+lk+l^2),0} & \oplus
\mathbbm{V}_{2(m-2)(k^2+lk+l^2),0} & \\
& \oplus \ldots \oplus
\mathbbm{V}_{2(k^2+lk+l^2),0} & \text{if $m$ is odd} \\
\end{array}
\right.
\end{equation}

With the help of the decompositions of $S^2(\mathfrak{p})$ and
$S^m(\mathfrak{p}^{\perp})$ as well as Schur's lemma it follows
that

\begin{equation}
\dim{W_m^h}=\left\{
\begin{array}{ll}
3 & \text{if $m$ is even} \\ 0 & \text{if $m$ is odd} \\
\end{array}
\right.\quad\text{and}\quad\dim{W_2^v}=3\:.
\end{equation}

Since $\mathfrak{p}^{\perp}$ is trivial with respect to
$U(1)_{k,l}$ and $N^{k,l}$ is not exceptional, $\mathfrak{p}$ and
$\mathfrak{p}^{\perp}$ have no non-trivial $U(1)_{k,l}$-submodule
in common and Assumption \ref{WEschAssumpt} is satisfied. We can
therefore apply Theorem \ref{EinsteinCohom1Solution} and Remark
\ref{WEschThmRem} to our situation and thus have proven the
following theorem.

\begin{Th}
\label{TheoremNklS1} In the situation of Theorem \ref{NklThm1}, let $N^{k,l}$ be generic and let $SU(3)/U(1)^2$ be a singular
orbit at $t=0$.

\begin{enumerate}
    \item There exists no $SU(3)$-invariant metric on $M$ which has holonomy Spin($7$).
    \item For any choice of $a_0,b_0,c_0\in\mathbb{R}\setminus\{0\}$ and
    $f_3,\lambda\in\mathbb{R}$, there exists a unique $SU(3)$-invariant
    Einstein metric on a tubular parameterize of $SU(3)/U(1)^2$ such that
    \begin{enumerate}
        \item $a(0)^2=a_0^2$, $b(0)^2=b_0^2$, $c(0)^2=c_0^2$,
        \item $f'''(0)=f_3$, and
        \item the Einstein constant is $\lambda$.
    \end{enumerate}
\end{enumerate}
\end{Th}

\subsection{$N^{1,0}$ as principal orbit}

In this situation, the only trivial $U(1)_{1,0}$-submodule of
$\mathfrak{m}$ again is $V_4$ and we have $f(0)=0$ as initial
condition. Since we restrict ourselves to diagonal metrics, we can
work as before with the system $(\ref{HolRedNkl})$. Therefore, we
can prove by the same arguments as in the previous case that $f$
vanishes and the metric is degenerate.

We apply the methods of \cite{Esch} in order to describe the
cohomogeneity-one Einstein metrics with the orbit structure of
this subsection. Since none of the $U(1)_{1,0}$-modules $V_1$,
$V_2$, and $V_3$ is trivial, Assumption \ref{WEschAssumpt} is
satisfied and we are in the situation of Theorem
\ref{EinsteinCohom1Solution}. If $k=1$ and $l=0$, the
decompositions of $S^2(\mathfrak{p})$ and
$S^m(\mathfrak{p}^\perp)$ which we have found specialize to

\begin{equation}
\begin{aligned}
S^2(\mathfrak{p}) & = \mathbbm{V}_{6,2} \oplus \mathbbm{V}_{0,4}
\oplus \mathbbm{V}_{-6,2} \\ & \quad \oplus \mathbbm{V}_{3,3}
\oplus \mathbbm{V}_{3,-1} \oplus \mathbbm{V}_{0,2} \oplus
\mathbbm{V}_{6,0} \\ & \quad \oplus \mathbbm{V}_{-3,3} \oplus
\mathbbm{V}_{3,1} \oplus 3\mathbb{R} \\
\end{aligned}
\end{equation}

and

\begin{equation}
S^m(\mathfrak{p}^\perp) = \left\{
\begin{array}{ll}
\mathbbm{V}_{2m,0} \oplus \mathbbm{V}_{2(m-2),0} \oplus \ldots
\oplus \mathbbm{V}_{4,0} \oplus \mathbb{R} & \text{if $m$ is even}
\\ \mathbbm{V}_{2m,0} \oplus \mathbbm{V}_{2(m-2),0} \oplus \ldots
\oplus \mathbbm{V}_{2,0} & \text{if $m$ is odd} \\
\end{array}
\right.
\end{equation}

$S^2(\mathfrak{p})$ contains a summand which is isomorphic to
$\mathbbm{V}_{6,0}$. Therefore, we have $\dim{W_m^h}=2$ if $m$ is
odd and $\geq 3$. We calculate $\dim{W_m^h}$ for all
$m\in\mathbb{N}_0$ and obtain

\begin{equation}
\dim{W_m^h} = \left\{\begin{array}{ll} 3 & \text{if $m$ is even}
\\ 0 & \text{if $m=1$} \\ 2 & \text{if $m\geq 3$ is odd} \\
\end{array}
\right.
\end{equation}

There are two initial conditions of $3^{rd}$ order which we can choose
freely. Since the summand $\mathbbm{V}_{6,0}$ is contained in
$V_1\otimes V_3$, the two parameters of $3^{rd}$ order describe how
the metric on the singular orbit, which is diagonal, changes into a
non-diagonal one. Analogously to the previous subsection we have proven
the following theorem.

\begin{Th}
\label{TheoremN10S1} In the situation of Theorem \ref{NklThm1},
let $N^{1,0}$ be the principal orbit and let $SU(3)/U(1)^2$ be a
singular orbit at $t=0$.

\begin{enumerate}
    \item There exists no $SU(3)$-invariant metric on $M$ which is diagonal
    with respect to the basis from page \pageref{qbasisAW} and has
    holonomy Spin($7$).
    \item For any choice of $a_0,b_0,c_0\in\mathbb{R} \setminus\{0\}$
    and $\beta,\widetilde{\beta},f_3,\lambda
    \in\mathbb{R}$, there exists a unique $SU(3)$-invariant
    Einstein metric on a tubular neighborhood of $SU(3)/U(1)^2$
    such that

    \begin{enumerate}
        \item $a(0)^2=a_0^2$, $b(0)^2=b_0^2$, $c(0)^2=c_0^2$,
        \item $\beta_{1,5}'''(0)=\beta$,
        $\beta_{1,6}'''(0)=\widetilde{\beta}$,
        \item $f'''(0)=f_3$, and
        \item the Einstein constant is $\lambda$.
    \end{enumerate}
\end{enumerate}
\end{Th}

\subsection{$N^{1,1}$ as principal orbit}
\label{N11S1Sec}

We make the same assumptions as in Theorem \ref{NklThm1} and thus
can work with the system (\ref{HolRedN11}). The isotropy algebra
$\mathfrak{k}$ of the $SU(3)$-action on the singular orbit is spanned
by $e_8$ and an arbitrary $x\in V_1\oplus V_4\setminus \{0\}$.
We assume that $x=e_7$ or equivalently that $f(0)=0$. Since the length
of the collapsing circle shall be $2\pi t + O(t^2)$ for small $t$, we need
$f'(0)\neq 0$. The differential equation for $f'$ contains the additional term
$-3\tfrac{(a_1 - a_2)^2}{a_1a_2}$, which is not there if the principal orbit
is generic. We therefore have $f'(0)\neq 0$ if and only if
$a_1(0)=a_0=-a_2(0)$ for an $a_0\in\mathbb{R}\setminus\{0\}$. The
values $b_0$ of $b(0)$ and $c_0$ of $c(0)$ can be chosen arbitrarily.
We thus have formulated an initial value problem for the first order
system (\ref{HolRedN11}).

We take a short look at the case where $x = \alpha e_1 + \beta e_2
+ \gamma e_7$ for some coefficients $\alpha$, $\beta$,
$\gamma\in\mathbb{R}$. Since we want $\lim_{t\rightarrow 0}
g_t(x,x)=0$, it follows that $x\in\{e_1, e_2, e_7\}$. There exists
a $\tau\in \text{Norm}_{SU(3)} U(1)_{1,1}$ which maps $e_1$ or
$e_2$ to $e_7$. Metrics which are described by solutions of
(\ref{HolRedN11}) with $a_1(0)=0$ or $a_2(0)=0$ can be therefore
be mapped by $\tau$ to metrics with $f(0)=0$. Diagonal metrics
with $a_1(0)=0$ or $a_2(0)=0$ can be mapped to non-diagonal
metrics, which we will not consider in the article. Nevertheless,
the above observation is a motivation to restrict ourselves to the
case $f(0)=0$.

We make a power series ansatz for (\ref{HolRedN11}) with the initial
values $a_0$, $b_0$, and $c_0$ and obtain

\begin{equation}
\label{N11S1PowSer}
\begin{array}{rclrrrrrrl}
a_1(t)\!\! & =\!\! & a_0 & - &
\tfrac{1}{2}\tfrac{a_0^2-b_0^2-c_0^2}{b_0c_0} t & + & \tfrac{1}{8}
\tfrac{3a_0^4-2a_0^2b_0^2-2a_0^2c_0^2-b_0^4+14b_0^2c_0^2-c_0^4}{
a_0b_0^2c_0^2}t^2 & +\!\! & \ldots \\ &&&&&&&& \\ a_2(t)\!\! &
=\!\! & -a_0 & - & \tfrac{1}{2}\tfrac{a_0^2-b_0^2-c_0^2}{b_0c_0}t
& - & \tfrac{1}{8}
\tfrac{3a_0^4-2a_0^2b_0^2-2a_0^2c_0^2-b_0^4+14b_0^2c_0^2-c_0^4}{
a_0b_0^2c_0^2} t^2 & +\!\! & \ldots \\ &&&&&&&& \\ b(t)\!\! &
=\!\! & b_0 & + & 0\cdot t & - &
\tfrac{1}{4}\tfrac{a_0^4-6a_0^2c_0^2-b_0^4+c_0^4}{a_0^2b_0c_0^2}t^2
& +\!\! & \ldots \\ &&&&&&&& \\ c(t)\!\! & =\!\! & c_0 & + &
0\cdot t & - &
\tfrac{1}{4}\tfrac{a_0^4-6a_0^2b_0^2+b_0^4-c_0^4}{a_0^2b_0^2c_0}t^2
& +\!\! & \ldots \\ &&&&&&&& \\ f(t)\!\! & =\!\! & 0 & + & 12t & +
& 0\cdot t^2 & +\!\! & \ldots \\
\end{array}
\end{equation}

The next issue which we will discuss is how the smoothness
conditions from Theorem \ref{SmoothExtension} translate into
conditions on the coefficients of the above power series. The
modules
$\mathfrak{p}=\mathbbm{V}_{6,0}\oplus\mathbbm{V}_{3,3}\oplus
\mathbbm{V}_{-3,3}$ and $\mathfrak{p}^\perp=\mathbbm{V}_{6,0}$
have one submodule in common and Assumption \ref{WEschAssumpt} is
thus not satisfied. Since we assume that the metric is diagonal,
it is nevertheless an element of $S^2(\mathfrak{p}) \oplus
S^2(\mathfrak{p}^\perp)$. We can therefore still check the
smoothness conditions by describing $W_m^h$ and $W_m^v$.
$S^2(\mathfrak{p})$ decomposes as

\begin{equation}
\begin{aligned}
S^2(\mathfrak{p}) & = \mathbbm{V}_{12,0} \oplus \mathbbm{V}_{6,6}
\oplus \mathbbm{V}_{-6,6} \\ & \quad\oplus \mathbbm{V}_{9,3}
\oplus \mathbbm{V}_{3,-3} \oplus \mathbbm{V}_{3,3} \oplus
\mathbbm{V}_{9,-3} \\ & \quad\oplus \mathbbm{V}_{0,6} \oplus
\mathbbm{V}_{6,0} \oplus 3\mathbb{R}\\
\end{aligned}
\end{equation}

For $S^m(\mathfrak{p}^\perp)$, we obtain

\begin{equation}
S^m(\mathfrak{p}^\perp) = \left\{
\begin{array}{ll}
\mathbbm{V}_{6m,0} \oplus \mathbbm{V}_{6(m-2),0} \oplus \ldots
\oplus \mathbbm{V}_{12,0} \oplus \mathbb{R} & \text{if $m$ is
even} \\ \mathbbm{V}_{6m,0} \oplus \mathbbm{V}_{6(m-2),0} \oplus
\ldots \oplus \mathbbm{V}_{6,0} & \text{if $m$ is odd} \\
\end{array}
\right.
\end{equation}

The dimensions of $W_m^h$ and $W_m^v$ can be calculated with the help
of Schur's lemma.

\begin{equation}
\dim{W_m^h} = \left\{\begin{array}{ll} 3 & \text{if $m=0$} \\ 5 &
\text{if $m\geq 2$ and even} \\ 2 & \text{if $m$ odd} \\
\end{array}
\right. \quad\quad \dim{W_m^v} = \left\{\begin{array}{ll} 1 &
\text{if $m=0$} \\ 3 & \text{if $m\geq 2$ and even} \\ 0 &
\text{if $m$ odd} \\
\end{array}
\right.
\end{equation}

The three dimensions of $W_0^h$ correspond to the coefficients of
$e^1\otimes e^1 + e^2\otimes e^2$, $e^3\otimes e^3 + e^4\otimes e^4$,
and $e^5\otimes e^5 + e^6\otimes e^6$. The two additional dimensions of
$W_m^h$ where $m$ is even and $\geq 2$ are caused by the submodule
$\mathbbm{V}_{12,0}$ of $S^2(V_1)$. They therefore describe the coefficients
of $e^1\otimes e^1 - e^2\otimes e^2$ and $e^1\otimes e^2 + e^2\otimes e^1$.
Since we consider only diagonal metrics, we can ignore the freedom which
we have for the second coefficient. The two dimensions of $W_m^h$ where
$m$ is odd describe $K$-equivariant maps $S^m(\mathfrak{p}^{\perp})
\rightarrow S^2(\mathfrak{p})$ which are non-zero only on
$V_2\otimes V_3$ and can be ignored for the same reasons as above.

The two additional dimensions of $W_m^v$ where $m$ is even and
$\geq 2$ describe the freedom of $g(\tfrac{\partial}{\partial t},
\tfrac{\partial}{\partial t})$ and $g(\tfrac{\partial}{\partial t},e_7)$ and
can be ignored, too (see Remark \ref{SmoothnessRemark}). As we have
remarked at the same place, the value of $|f'(0)|$ has to be chosen in
such a way that the length of the collapsing circle is $2\pi t + O(t^2)$
for small $t$. Since $\exp{(e_7 t)}$ intersects $U(1)_{1,1}$ at
$t=\tfrac{\pi}{3}$ for the first time, the length of the collapsing circle
is

\begin{equation*}
\int_0^{\tfrac{\pi}{3}} \sqrt{g(e_7,e_7)} ds = \tfrac{\pi}{3}|f(t)|\:.
\end{equation*}

We therefore obtain $|f'(0)|=6$ as a smoothness condition. By
translating our statements on $W_m^h$ and $W_m^v$ into conditions
on power series expansion of the metric we finally see that an
analytic diagonal metric of type (\ref{MetricsN11}) is smooth if
and only if

\begin{itemize}
    \item $a_1^2(0) = a_2^2(0)$,
    \item $a_1^2$, $a_2^2$, $b^2$, and $c^2$ are even functions,
    \item $f(0)=0$, $|f'(0)|=6$, and
    \item $f$ is odd.
\end{itemize}

The power series (\ref{N11S1PowSer}) does not satisfy these
conditions. Nevertheless, there is a method to obtain smooth
cohomogeneity-one metrics whose holonomy is contained in
$\text{Spin}(7)$ from (\ref{N11S1PowSer}). Let

\begin{equation}
\label{hZ2}
h:=\left(\,\begin{array}{ccc} i & 0 & 0 \\ 0 & -i & 0 \\ 0 & 0 & 1
\\
\end{array}\,\right)\in SU(3)\:.
\end{equation}

$h$ acts by $h.gU(1)_{1,1}:=hgh^{-1}U(1)_{1,1}$ on $N^{1,1}$ and
stabilizes the $G_2$-struc\-tures which are determined by a basis
of type (\ref{CBasisN11}). Since $h^2\in U(1)_{1,1}$, $N^{1,1}$ is
a double cover of the quotient of $SU(3)$ by the group which is
ge\-nerated by $U(1)_{1,1}$ and $h$. We denote this quotient
simply by $N^{1,1}/\mathbb{Z}_2$. On $SU(3)/U(1)^2$, $h$ acts
trivially. We divide the cohomogeneity-one manifold with principal
orbit $N^{1,1}$ and singular orbit $SU(3)/U(1)^2$ by the group
which is generated by the simultaneous action of $h$ on all
orbits. Any (smooth or non-smooth) Spin($7$)-structure on the old
manifold which is induced by (\ref{CBasisN11}) is mapped by the
quotient map to a new one.

It is easy to see that any circle which was wrapped $r$ times
around the origin of the old normal space is wrapped $2r$ times
around the origin of the new normal space. We reconsider the
arguments which we have made and see that in this new situation we
have to require $|f'(0)|=12$ instead of $|f'(0)|=6$ in order to
make the metric smooth at the singular orbit. Since $U(1)^2$ now
acts on $\mathfrak{p}^\perp$ as $\mathbbm{V}_{12,0}$ rather than
$\mathbbm{V}_{6,0}$, we have

\begin{equation}
\dim{W_m^h} = \left\{\begin{array}{ll} 3 & \text{if $m$ even} \\ 2
& \text{if $m$ odd} \\
\end{array}
\right.
\end{equation}

Since the values of $\dim{W_m^h}$ have changed, the smoothness
conditions on the functions $a_1$, $a_2$, $b$, and $c$ have
changed, too. The meaning of the three dimensions in the even case
is the same as before and the two dimensions in the odd case now
correspond to the components of the metric in $S^2(V_1)$. The
dimensions of $W_m^v$ stay the same and we have found the
following new smoothness conditions.

\begin{itemize}
    \item $a_1^2(t)=a_2^2(-t)$,
    \item $b^2$ and $c^2$ are even,
    \item $f(0)=0$, $|f'(0)|=12$, and
    \item $f$ is odd.
\end{itemize}

In particular, these conditions are satisfied if
$a_1(t)=-a_2(-t)$, $b(t)=b(-t)$ and $c(t)=c(-t)$. The first
coefficients of (\ref{N11S1PowSer}) obviously satisfy this new set
of conditions.

By an explicit calculation we can prove that any power series
solution of (\ref{HolRedN11}) with the same initial conditions as
in this subsection is uniquely determined by $a_0$, $b_0$, and
$c_0$. The system (\ref{HolRedN11}) is preserved if we replace
$(a_1(t),a_2(t),b(t),c(t),f(t))$ by
$(-a_2(-t),-a_1(-t),b(-t),c(-t), -f(-t))$.\linebreak Therefore,
both sets of functions are the unique solutions of the same
initial value problem and the power series thus satisfies the
smoothness conditions. $h$ acts trivially on
$\mathfrak{p}^{\perp}$ and non-trivially on $V_1$, $V_2$, and
$V_3$. For this reason, Assumption \ref{WEschAssumpt} is satisfied
and we can show with the help of Theorem
\ref{EinsteinCohom1Solution} that the series converges near the
singular orbit. Moreover, we can use the dimensions of $W_m^h$ and
$W_m^v$ to calculate the number of cohomogeneity-one Einstein
metrics near the singular orbit.

We take another look at the power series (\ref{N11S1PowSer}). The
condition $a_1(t) + a_2(t) = 0$ can only be satisfied if
$a_0^2=b_0^2 + c_0^2$. In that situation, we obtain the metrics of
Bazaikin and Malkovich \cite{Baz2}, which have holonomy $SU(4)$.
We finally have proven the following theorem.

\begin{Th}
Let $M$ be a cohomogeneity-one manifold whose principal orbit is $N^{1,1}/
\mathbb{Z}_2$ where $\mathbb{Z}_2$ is generated by (\ref{hZ2}). We assume
that $M$ has exactly one singular orbit of type $SU(3)/U(1)^2$, which is at
$t=0$.

\begin{enumerate}
    \item For any $a_0,b_0,c_0\in\mathbb{R}\setminus\{0\}$ there exists a
    unique $SU(3)$-invariant parallel Spin($7$)-structure $\Omega$ on a
    tubular simultaneous of $SU(3)/U(1)^2$ which is determined by a
    basis of type (\ref{CBasisN11}) and satisfies $f(0)=0$,
    $a_1(0)=-a_2(0)=a_0$, $b(0)=b_0$, and $c(0)=c_0$.

    \item The metric associated to $\Omega$ is a diagonal metric of type
    (\ref{MetricsN11}). If $a_0^2 = b_0^2 + c_0^2$, its holonomy is
    $SU(4)$.

    \item For any choice of $a_0,b_0,c_0\in\mathbb{R}\setminus\{0\}$
    and $a_1,\beta,f_3,\lambda\in\mathbb{R}$, there
    exists a unique $SU(3)$-invariant Einstein metric on a tubular
    neighborhood of $SU(3)/U(1)^2$ such that

    \begin{enumerate}
        \item $f(0)=0$, $a_1(0)^2=a_2(0)^2=a_0^2$, $b(0)^2=b_0^2$, $c(0)^2=c_0^2$,
        \item $(a_1-a_2)'(0)=a_1$, $\beta_{1,2}'(0)=\beta$,
        \item $f'''(0)=f_3$, and
        \item the Einstein constant is $\lambda$.
    \end{enumerate}
\end{enumerate}
\end{Th}

\begin{Rem}
A power series ansatz for the metrics with special holonomy from
the above theorem has also been made in Kanno, Yasui \cite{KanI}.
These metrics were also investigated by Bazaikin and Malkovich
\cite{Baz}, \cite{Baz2}. The\linebreak authors study the
smoothness and completeness of the metrics and describe the family
of metrics with holonomy $SU(4)$ explicitly. They also prove that
the holonomy of their metrics is always $SU(4)$ except in the
limiting case where $b_0\rightarrow 0$ and the metric degenerates
into the hyperk\"ahler metric of Calabi \cite{Calabi} on
$T^{\ast}\mathbb{CP}^2$. Our methods for the calculation of the
smoothness conditions and for the proof of the holonomy reduction
in the case $a_1(t) + a_2(t) = 0$ are new contributions of the
author.
\end{Rem}

\section{$S^5$ as singular orbit}
\label{S5SingSection}

$S^5$ is a possible singular orbit only if $k\cdot l\cdot (-k-l)=0$. We
assume in this section that $(k,l)=(1,-1)$, since this implies
the initial conditions $a(0)=0$ and $b(0),c(0),f(0)\neq 0$. The only
difference to the case $(k,l)=(1,0)$ is that the off-diagonal
coefficients of the metric are contained in $V_2\otimes V_3\oplus V_3\otimes
V_2$ instead of $V_1\otimes V_3\oplus V_3\otimes V_1$ and
hence will be denoted by $\beta_{3,5}$ and $\beta_{3,6}$ instead of
$\beta_{1,5}$ and $\beta_{1,6}$. We want the right-hand side of the
second and third equation of (\ref{HolRedNkl}) to converge to a
real number for $t\rightarrow 0$. This is only possible if
$b(0)^2 = c(0)^2$. Since we may replace $c$ by $-c$, $f$ by $-f$, and
$t$ by $-t$ without changing the system (\ref{HolRedNkl}), we can
assume that $b(0)=c(0)$. We shortly denote $b(0)$ by $b_0$
and $f(0)$ by $f_0$. By a power series ansatz, we obtain the following
solution of the system (\ref{HolRedNkl}).

\begin{equation}
\label{N10S2PowSer}
\begin{array}{rclrrrrrrrrl}
a(t)\!\! & =\!\! & 0 & + & 2t & + & 0\cdot t^2 & - &
\tfrac{1}{27}\tfrac{36b_0^2-f_0^2}{b_0^4} t^3 & +\!\! & \ldots \\
&&&&&&&&&& \\ b(t)\!\! & =\!\! & b_0 & - & \tfrac{1}{6}\tfrac{f_0}{b_0}t &
+ & \tfrac{1}{72}\tfrac{72b_0^2-5f_0^2}{b_0^3}t^2 & + &
\tfrac{1}{6480}\tfrac{f_0(504b_0^2-167f_0^2)}{b_0^5} t^3 & +\!\! &
\ldots \\ &&&&&&&&&& \\ c(t)\!\! & =\!\! & b_0 & + &
\tfrac{1}{6}\tfrac{f_0}{b_0}t & + &
\tfrac{1}{72}\tfrac{72b_0^2-5f_0^2}{b_0^3}t^2 & - &
\tfrac{1}{6480}\tfrac{f_0(504b_0^2-167f_0^2)}{b_0^5} t^3 & +\!\! &
\ldots \\ &&&&&&&&&& \\ f(t)\!\! & =\!\! & f_0 & + & 0\cdot t & + &
\tfrac{1}{6}\tfrac{f_0^3}{b_0^4}t^2 & + & 0\cdot t^3 & +\!\! & \ldots
\\
\end{array}
\end{equation}

Our next step is to decompose $S^m(\mathfrak{p}^{\perp})$ and
$S^2(\mathfrak{p})$ into $SU(2)$-submodules. This is necessary
in order to deduce the smoothness conditions. We denote the
complex irreducible representation of $SU(2)$ with weight $r$ by
$\mathbbm{V}_r^{\mathbb{C}}$. The real irreducible representation
with the same weight we denote by $\mathbbm{V}_r^{\mathbb{R}}$.
We recall that $\dim{\mathbbm{V}_r^{\mathbb{R}}}=r+1$ if $r$
is even and $\dim{\mathbbm{V}_r^{\mathbb{R}}}=2(r+1)$ if $r$
is odd.

The orbits of the $SU(2)$-action on the three-dimensional space
$\mathfrak{p}^{\perp}$ are spheres. Therefore, it is isomorphic
to $\mathbbm{V}_2^{\mathbb{R}}$. In order to decompose
$S^m(\mathfrak{p}^{\perp})$, we first consider the space
$S^m_{\mathbb{R}}(\mathbb{C}^2)$ of all homogeneous
polynomials of $m^{th}$ order with real coefficients
on $\mathbb{C}^2$. The subscript $\mathbb{R}$ means
that we consider $\mathbb{C}^2$ as a four-dimensional
real vector space and we have $\dim{S^2_{\mathbb{R}}
(\mathbb{C}^2)} = 10$, $\dim{S^3_{\mathbb{R}}
(\mathbb{C}^2)} = 20$, etc. The complexification
$S^m_{\mathbb{R}}(\mathbb{C}^2) \otimes \mathbb{C}$
consists of all polynomials depending on
$z_1,z_2,\overline{z_1},\overline{z_2}$. Let
$\mathbbm{V}_{s,\overline{m-s}}$ be the subspace of all
trace-free polynomials depending on $s$ complex and
$m-s$ conjugate complex variables. $S^m_{\mathbb{R}}(\mathbb{C}^2)
\otimes\mathbb{C}$ and $\mathbbm{V}_{s,\overline{m-s}}$
are obviously $SU(2)$-modules and we have the
following decomposition

\begin{equation}
S^m_{\mathbb{R}}(\mathbb{C}^2)\otimes \mathbb{C}=
\bigoplus_{p=0}^{\lfloor \tfrac{m}{2}
\rfloor}\:\bigoplus_{s=0}^{m-2p} \mathbbm{V}_{s,
\overline{m-2p-s}}\:.
\end{equation}

The submodules $\mathbbm{V}_{s,\overline{m-2p-s}}$ are
irreducible and by calculating their dimension we see
that

\begin{equation}
S^m_{\mathbb{R}}(\mathbb{C}^2)\otimes \mathbb{C}=
\bigoplus_{p=0}^{\lfloor \tfrac{m}{2} \rfloor} (m- 2p
+1)\mathbbm{V}_{m-2p}^{\mathbb{C}}\:,
\end{equation}

where $(m- 2p +1)\mathbbm{V}_{m-2p}^{\mathbb{C}}$ denotes the
direct sum of $m- 2p +1$ copies of
$\mathbbm{V}_{m-2p}^{\mathbb{C}}$. $S^m_{\mathbb{R}}(
\mathbb{C}^2)$ consists of exactly those elements of
$S^m_{\mathbb{R}}(\mathbb{C}^2) \otimes\mathbb{C}$ which are
invariant with respect to the conjugation map $\tau$. $\tau$ maps
$\mathbbm{V}_{s,\overline{m-2p-s}}$ into
$\mathbbm{V}_{m-2p-s,\overline{s}}$ and vice versa. If $m$ is even
and $s\neq \tfrac{m-2p}{2}$, the subspace of
$\mathbbm{V}_{s,\overline{m-2p-s}}
\oplus\mathbbm{V}_{m-2p-s,\overline{s}}$ is invariant under
$\tau$ and decomposes into two real submodules of the same dimension.
Both submodules are irreducible and equivalent to
$\mathbbm{V}_{m-2p}^{\mathbb{R}}$.
$\mathbbm{V}_{\tfrac{m-2p}{2},\overline{\tfrac{m-2p}{2}}}$ is a
real module which is also isomorphic to
$\mathbbm{V}_{m-2p}^{\mathbb{R}}$. If $m$ is odd, the subspace of
$\mathbbm{V}_{s,\overline{m-2p-s}}\oplus
\mathbbm{V}_{m-2p-s,\overline{s}}$ is invariant under $\tau$,
irreducible, and equivalent to
$\mathbbm{V}_{m-2p}^{\mathbb{R}}$. All in all, we have

\begin{equation}
\label{N10S2Decomp0}
S^m_{\mathbb{R}}(\mathbb{C}^2) = \left\{\begin{array}{ll}
\bigoplus\limits_{p=0}^{\tfrac{m}{2}} (2p +1)
\mathbbm{V}_{2p}^{\mathbb{R}} & \text{if $m$ is even}\\ & \\
\bigoplus\limits_{p=0}^{\tfrac{m-1}{2}} (p+1)
\mathbbm{V}_{2p+1}^{\mathbb{R}} & \text{if $m$ is odd}\\
\end{array}\right.
\end{equation}

$\mathbbm{V}_2^{ \mathbb{R}}$ can be identified with
$\mathbbm{V}_{1,\overline{1}}$. $S^m( \mathbbm{V}_2^{\mathbb{R}})$
therefore is a submodule of $\mathbbm{V}_{m,\overline{m}}$.
With the help of the above considerations we obtain

\begin{equation}
\label{N10S2Decomp}
S^m(\mathbbm{V}_2^{\mathbb{R}}) = \bigoplus_{p=0}^{\lfloor
\tfrac{m}{2}\rfloor} \mathbbm{V}_{m-2p,\overline{m-2p}} =
\bigoplus_{p=0}^{\lfloor \tfrac{m}{2}\rfloor}
\mathbbm{V}_{2m-4p}^{\mathbb{R}}\:.
\end{equation}

We are now able to calculate the dimension of $W_m^h$ and
$W_m^v$. By a short calculation we see that

\begin{equation}
\mathfrak{p} = \mathbbm{V}_1^{\mathbb{C}} \oplus
\mathbbm{V}_0^{\mathbbm{R}}\:.
\end{equation}

With the help of (\ref{N10S2Decomp0}) we conclude that

\begin{equation}
S^2(\mathfrak{p}) = S^2(\mathbbm{V}_1^{\mathbb{C}}) \oplus
\left(\mathbbm{V}_1^{\mathbb{C}}\otimes
\mathbbm{V}_0^{\mathbb{R}}\right)
\oplus S^2(\mathbbm{V}_0^{\mathbb{R}}) =
3\mathbbm{V}_2^{\mathbb{R}} \oplus \mathbbm{V}_1^{\mathbb{C}}
\oplus 2\mathbbm{V}_0^{\mathbb{R}}
\end{equation}

and finally obtain

\begin{equation}
\dim{W_m^h} = \left\{\begin{array}{ll} 2 & \text{if $m$ is even}
\\ 3 & \text{if $m$ is odd} \\
\end{array}
\right.
\end{equation}

We interpret these numbers and start with the case $m=0$. Any
$SU(2)$-invariant metric on the singular orbit is diagonal and
satisfies $b(0)^2=c(0)^2$. $\dim{W_0^h}=2$ therefore simply
means that we can choose the initial values $b_0$ and $f_0$
freely. Let $b_m$, $c_m$, and $f_m$ denote the $m^{th}$
coefficients of the power series for $b$, $c$, and $f$. Analogously
to the case $m=0$, $\dim{W_m^h} = 2$ for even $m$ means
that $b_m = c_m$.

There is a suitable $h\in SU(2)$ which acts on
$\text{span}(e_3,e_4,e_5,e_6) \subseteq \mathfrak{p}$ in
the same way as $j\in Sp(1)$ by right-multiplication on
$\mathbb{H}$. Since $\mathfrak{p}^{\perp}$ is
three-dimensional, $j$ acts as an rotation around an angle of $\pi$.
Moreover, it is a rotation around an axis perpendicular to
$\tfrac{\partial}{\partial t}$ and thus turns
$\tfrac{\partial}{\partial t}$ into $-\tfrac{\partial}{\partial t}$.
Since the metric $g$ in invariant under $h$, it follows that

\begin{equation*}
\begin{aligned}
g_t(e_3,e_3) & = g_{-t}(e_5,e_5)\\
g_t(e_3,e_5) & = - g_{-t}(e_5,e_3)\\
g_t(e_3,e_6) & = - g_{-t}(e_5,e_4) = -g_{-t}(e_3,e_6)\\
g_t(e_7,e_7) & = g_{-t}(e_7,e_7)\\
\end{aligned}
\end{equation*}

This translates into $b_m=-c_m$ and $f_m=0$ if $m$ is odd. The
space of all $U(1)_{1,-1}$-invariant elements of $S^2(\mathfrak{p})$
satisfying these conditions has dimension $3$ which equals $\dim{W_m^h}$.
Therefore, there are no further conditions on the horizontal part and
the horizontal part of an analytic metric is smooth if and only if

\begin{equation}
\label{S2Smooth}
b(t)  =  c(-t)\quad \text{and} \quad f, \beta_{3,5},\beta_{3,6}\:\:\text{are even.}
\end{equation}

We turn to the vertical component of the metric, which is determined by $a$.
It follows from (\ref{N10S2Decomp}) that

\begin{equation}
S^2(\mathfrak{p}^\perp) = \mathbbm{V}_4^{\mathbb{R}} \oplus
\mathbbm{V}_0^{\mathbbm{R}}\:.
\end{equation}

Since $\dim{\text{Hom}_{SU(2)}(\mathbbm{V}_4^{\mathbb{R}},
\mathbbm{V}_4^{\mathbb{R}})} = 1$, we have

\begin{equation}
\dim{W_m^v} = \left\{\begin{array}{ll} 1 & \text{if $m=0$} \\ 0 &
\text{if $m$ is odd} \\ 2 & \text{if $m\geq 2$ and even} \\
\end{array}
\right.
\end{equation}

For similar reasons as in Subsection \ref{N11S1Sec}, this means that
$a$ has to be an odd function. The only missing smoothness
condition is that the length $\ell(t)$ of any great circle of the collapsing
sphere $SU(2)/U(1)_{1,-1}$ has to be $2\pi t + O(t^2)$ for small $t$.
The Lie group which is generated by $e_1$ intersects $U(1)_{1,-1}$
twice. Therefore, $\ell(t)=\sqrt{g_t(e_1,e_1)}\:\pi = |a(t)|\pi$. By
the same argument as in Subsection \ref{N11S1Sec}, it follows that
$|a'(0)|$ has to be $2$, which is indeed the case.

As in the previous case, we make an explicit calculation and prove
that for any choice of $b_0,f_0\in\mathbb{R}\setminus\{0\}$ there
exists a unique power series solution of (\ref{HolRedNkl}).
$(a(t),b(t),c(t),f(t)) \mapsto (-a(-t),c(-t), b(-t),f(-t))$ is a symmetry
of (\ref{HolRedNkl}) if $(k,l)=(1,-1)$. As in the previous subsection,
it follows that the power series satisfies the smoothness conditions.

Unfortunately, Assumption \ref{WEschAssumpt} is not satisfied,
since $\mathfrak{p}$ and $\mathfrak{p}^{\perp}$ both contain
a trivial submodule, namely $\text{span}(e_7)$ or
$\text{span}(\tfrac{\partial}{\partial t})$ respectively. Nevertheless,
we always have due to the choice of our coordinates $g(e_7,
\tfrac{\partial}{\partial t})=0$. Moreover, we have $\text{Ric}(e_7,
\tfrac{\partial}{\partial t})=0$ (see Grove, Ziller \cite{Gro}). Therefore,
$g$, $\text{Ric}$, and the steps of the Picard iteration in
Eschenburg, Wang \cite{Esch} do not leave the space $S^2(\mathfrak{p}) \oplus
S^2(\mathfrak{p}^{\perp})$. The arguments of the proof in
\cite{Esch} therefore remain unchanged and we can conclude that the
power series converges and are able to determine the number of the
Einstein metrics.

\begin{Th}
Let $M$ be a cohomogeneity-one manifold whose principal orbit is
$N^{1,-1}$. We assume that $M$ has exactly one singular orbit of
type $S^5$, which is at $t=0$.

\begin{enumerate}
    \item In this situation, any cohomogeneity-one metric satisfies
    $a(0)=0$.

    \item For any $b_0,f_0\in\mathbb{R}\setminus\{0\}$ there exists a
    unique $SU(3)$-invariant parallel Spin($7$)-structure $\Omega$ on a
    tubular simultaneous of $S^5$ which is determined by a
    basis of type (\ref{CBasisNkl}) and satisfies $b(0)=c(0)=b_0$
    and $f(0)=f_0$. The holonomy of the associated metric is all
    of Spin($7$).

    \item For any choice of $b_0,f_0\in\mathbb{R}\setminus\{0\}$
    and $b_1,\beta, \widetilde{\beta},$ $a_3,\lambda\in\mathbb{R}$, there
    exists a unique $SU(3)$-invariant Einstein metric on a tubular
    neighborhood of $S^5$ such that

    \begin{enumerate}
        \item $b(0)^2=c(0)^2=b_0^2$, $f(0)^2=f_0^2$,
        \item $(b-c)'(0)=b_1$, $\beta_{3,5}'(0)=\beta$,
        $\beta_{3,6}'(0)=\widetilde{\beta}$,
        \item $a'''(0)=a_3$, and
        \item the Einstein constant is $\lambda$.
    \end{enumerate}
\end{enumerate}
\end{Th}

\begin{Rem}
\begin{enumerate}
    \item The metrics with singular orbit $S^5$ and holonomy\linebreak
    Spin($7$) were also considered by Cveti\v{c} et al.
    \cite{Cve}. The discussion of the smoothness conditions,
    the convergence of the power series and the existence of
    the Einstein metrics are new results.
    \item If the principal orbit is $N^{1,-1}$, the singular orbit
    can also be a space of type $SU(3)/SO(3)$. However, it is
    impossible that the metric on $SU(3)/SO(3)$ is positive and
    the volume of $SO(3)/U(1)_{1,-1}$ shrinks to zero if
    the metric is diagonal with respect to $(e_i)_{i=1\leq i\leq 7}$.
    Although it is possible to construct spaces with singular
    orbit $SU(3)/$ $SO(3)$ by considering non-diagonal metrics,
    we will not investigate this case further.
\end{enumerate}
\end{Rem}

\section{$\mathbb{CP}^2$ as singular orbit}

We assume that the principal orbit is a generic Aloff-Wallach space
$N^{k,l}$. In this situation, the isotropy algebra $\mathfrak{k}$ of
the $SU(3)$-action on $\mathbb{CP}^2$ is either
$\mathfrak{u}(1)_{k,l}\oplus V_1 \oplus V_4$,
$\mathfrak{u}(1)_{k,l}\oplus V_2 \oplus V_4$, or
$\mathfrak{u}(1)_{k,l} \oplus V_3 \oplus V_4$. As we have remarked
in Convention \ref{NklConv2}, it suffices to work with the case
$\mathfrak{k}=\mathfrak{u}(1)_{k,l}\oplus V_1\oplus V_4$ if we
consider all pairs $(k,l)$ of coprime integers with $k\geq l$.

If $(k,l)=(1,-1)$, $\mathfrak{u}(1)_{k,l}\oplus
V_1\oplus V_4$ is not a possible choice of $\mathfrak{k}$
since $K/U(1)_{1,-1}$ is diffeomorphic to $S^2\times S^1$.
Nevertheless, the two related cases $(k,l)\in\{(1,0),(0,-1)\}$ are
still possible.

If $k=l=1$ and the metric is diagonal,
$\mathfrak{u}(1)_{1,1}\oplus V_1 \oplus V_4$,
$\mathfrak{u}(1)_{1,1}\oplus V_2 \oplus V_4$, and
$\mathfrak{u}(1)_{1,1} \oplus V_3 \oplus V_4$ are still the
only possibilities for $\mathfrak{k}$. If the complement of
$\mathfrak{u}(1)_{1,1}$ was transversely embedded into
$V_1\oplus V_2\oplus V_3$, the null space of the
degenerate metric $g_0\in S^2(\mathfrak{m})$ could not be
$\mathfrak{k}$. Moreover,
$\mathfrak{u}(1)_{1,1}\oplus V_2 \oplus V_4$ and
$\mathfrak{u}(1)_{1,1} \oplus V_3 \oplus V_4$ can be
obtained by the action of an element of
$\text{Norm}_{SU(3)}U(1)_{1,1}$ from each other. We therefore
have to consider only one of these cases. All in all, we have to
consider the following three initial values problems.

\begin{enumerate}
    \item The system (\ref{HolRedNkl}) with $a(0)=f(0)=0$ and $(k,l)\not\in\{(1,-1),(1,1),$ $(1,-2),(2,-1)\}$,
    \item The system (\ref{HolRedN11}) with $a_1(0)=a_2(0)=f(0)=0$,
    \item The system (\ref{HolRedN11}) with $b(0)=f(0)=0$.
\end{enumerate}

We make a power series ansatz for all of the above initial value problems and
start with the first one. As in Section \ref{S5SingSection}, we can assume
without loss of generality that $b(0)=c(0)=:b_0$. The Taylor expansion of
any solution of our initial value problem begins with

\begin{equation}
\label{PowSerS3Zkl}
\begin{array}{l}
\begin{array}{rclrrrrrrrrl}
a(t) & = & 0  & + & t & + & 0\cdot t^2 & - &
\tfrac{1}{24b_0^2}\tfrac{12\Delta+q (k+l)}{\Delta} t^3 & + &
0\cdot t^4 & +\ldots \\
\end{array} \\ \\
\begin{array}{rclrrrrrr}
b(t) & = & b_0 & + & 0\cdot t & + &
\tfrac{1}{6b_0}\tfrac{4k+5l}{k+l} t^2 & + & 0\cdot t^3 \\
\end{array} \\ \\
\begin{array}{rrrl}
\quad\quad\quad\quad\quad & + & \tfrac{1}{288b_0^3}
\tfrac{(-104k^2 - 224kl - 140l^2)\Delta + q(-k^3 - k^2l + kl^2 +
l^3)}{\Delta^2} t^4 & + \ldots\\
\end{array} \\ \\
\begin{array}{rclrrrrrr}
c(t) & = & b_0 & + & 0\cdot t & + &
\tfrac{1}{6b_0}\tfrac{5k+4l}{k+l} t^2 & + & 0\cdot t^3 \\
\end{array} \\ \\
\begin{array}{rrrl}
\quad\quad\quad\quad\quad & + & \tfrac{1}{288b_0^3}
\tfrac{(-140k^2 - 224kl - 104l^2)\Delta + q(k^3 + k^2l - kl^2 -
l^3)}{\Delta^2} t^4 & + \ldots\\
\end{array} \\ \\
\begin{array}{rclrrrrrrrrl}
f(t) & = & 0 & + & \tfrac{2\Delta}{k+l}t & + & 0\cdot t^2 & + &
\tfrac{q}{6b_0^2} t^3 & + & 0\cdot t^4 & + \ldots\\
\end{array} \\
\end{array}
\end{equation}

The parameter $q$ of third order can be chosen freely. Any
solution of (\ref{HolRedN11}) with $a_1(0)=a_2(0)=f(0)$ also has
to satisfy $b(0)^2=c(0)^2$. We again can assume that
$b(0)=c(0)=:b_0$, since $(c(t),f(t))\mapsto (-c(-t),-f(-t))$ is a
symmetry of the system (\ref{HolRedN11}), and obtain the following
Taylor expansion.

\begin{equation}
\label{PowSerN11S3}
\begin{array}{l}
\begin{array}{rclrrrrrrrl}
a_1(t) & = & 0 & + & t & + & 0\cdot t^2 & + &
\tfrac{q_1}{6b_0^2}t^3 & + & 0\cdot t^4 \\
\end{array} \\ \\
\begin{array}{rrrl}
\quad\quad\quad\quad\quad\: & + & \tfrac{2q_2^1 -
3q_1q_2 - 3q_2^2 - 3q_1 - 18q_2}{60 b_0^4}t^5 & + \ldots \\
\end{array} \\ \\
\begin{array}{rclrrrrrrrl}
a_2(t) & = & 0 & + & t & + & 0\cdot t^2 & + &
\tfrac{q_2}{6b_0^2}t^3 & + & 0\cdot t^4 \\
\end{array} \\ \\
\begin{array}{rrrl}
\quad\quad\quad\quad\quad\: & + & \tfrac{-3q_1^2 -
3q_1q_2 + 2q_2^2 -18q_1 -3q_2}{60 b_0^4}t^5 & + \ldots \\
\end{array} \\ \\
\begin{array}{rclrrrrrrrrrrl}
b(t)\:\:\: & = & b_0\!\!\! & + & 0\cdot t & + & \tfrac{3}{4b_0}t^2
& + & 0\cdot t^3 & - & \tfrac{39}{96 b_0^3}t^4 & + & 0\cdot t^5 &
+ \ldots \\
\end{array} \\ \\
\begin{array}{rclrrrrrrrrrrl}
c(t)\:\:\: & = & b_0\!\!\! & + & 0\cdot t & + & \tfrac{3}{4b_0}t^2
& + & 0\cdot t^3 & - & \tfrac{39}{96 b_0^3}t^4 & + & 0\cdot t^5 &
+\ldots \\
\end{array} \\ \\
\begin{array}{rclrrrrrrrl}
f(t)\:\:\: & = & 0\!\! & + & 3t & + & 0\cdot t^2 & - &
\tfrac{6+q_1+q_2}{2b_0^2} t^3 & + & 0\cdot t^4 \\
\end{array} \\ \\
\begin{array}{rrrl}
\quad\quad\quad\quad\quad\: & + &\tfrac{2q_1^2 + 7q_1q_2 + 2q_2^2
+ 27q_1 + 27q_2 +90}{20b_0^4} t^5 & +\ldots \\
\end{array}
\end{array}
\end{equation}

Analogously to the previous case, there are two parameters $q_1$ and
$q_2$ which can be chosen freely. The functions $b$ and $c$ coincide
up to fifth order. Later on, we will prove that actually $b(t)=c(t)$ for all
values of $t$. Next we consider the equations (\ref{HolRedN11}) under
the assumption that $b(0)=f(0)=0$. We necessarily have
$c(0)^2=a_1(0)^2=a_2(0)^2$. Let $a_0:=a_1(0)$. Since there are four
possibilities for the signs of $a_2(0)$ and $c(0)$, there are four kinds of
initial value problems. Because of the symmetry of (\ref{HolRedN11})
which we have used in the previous case, we can assume that the sign
of $a_1(0)$ and $c(0)$ is the same. Therefore, the only two subcases which we have
to consider are $a_1(0)=a_2(0)$ and $a_1(0)=-a_2(0)$. The initial condition
$a_1(0)=a_2(0)$ yields the following power series.

\begin{equation}
\label{PowSerN11S3B}
\begin{array}{l}
\begin{array}{rclrrrrrrrrrrl}
a_1(t) & = & a_0 & + & 0\cdot t & + & \tfrac{1}{a_0}t^2 & + &
0\cdot t^3 & - & \tfrac{q+21}{24a_0^3} t^4 & + & 0\cdot t^5 & +
\ldots \\
\end{array} \\ \\
\begin{array}{rclrrrrrrrrrrl}
a_2(t) & = & a_0 & + & 0\cdot t & + &
\tfrac{1}{a_0}t^2 & + & 0\cdot t^3 & - & \tfrac{q+21}{24a_0^3} t^4
& + & 0\cdot t^5 & + \ldots \\
\end{array} \\ \\
\begin{array}{rclrrrrrrrl}
b(t)\:\: & = & 0\:\:\: & + & t & + & 0\cdot t^2 & + &
\tfrac{q}{6a_0^2} t^3 & + & 0\cdot t^4 \\
\end{array} \\ \\
\begin{array}{rrrl}
\quad\quad\quad\quad\quad\:\:\: & - & \tfrac{8q^2 + 42q + 9}{120
a_0^4}t^5 & + \ldots \\
\end{array} \\ \\
\begin{array}{rclrrrrrrrrrrl}
c(t)\:\: & = & a_0 & + & 0\cdot t & + &
\tfrac{1}{2a_0}t^2 & + & 0\cdot t^3 & + & \tfrac{q-6}{24a_0^3}t^4
& + & 0\cdot t^5 & +\ldots \\
\end{array} \\ \\
\begin{array}{rclrrrrrrrl}
f(t)\:\: & = & 0\:\: & - & 6t & + & 0\cdot t^2 & + &
\tfrac{2(q+3)}{a_0^2}t^3 & + & 0\cdot t^4 \\
\end{array} \\ \\
\begin{array}{rrrl}
\quad\quad\quad\quad\quad\:\:\: & - &\tfrac{11q^2 + 54q
+123}{10a_0^4} t^5 & +\ldots \\
\end{array}
\end{array}
\end{equation}

where $q$ is a free parameter. Later on it will be proven that $a_1(t)=a_2(t)$
for all $t$. Next, we study (\ref{HolRedN11})
under the assumption that $b(0)=f(0)=0$ and $a_1(0)=-a_2(0)$. We
obtain a system of quadratic equations for the first derivatives
$(a_1'(0),\ldots,f'(0))$. The only two meaningful solutions
of that system are $(0,0,1,0,6)$ and $(0,0,-1,0,6)$. $(a_1(t),a_2(t),b(t),c(t),f(t))
\mapsto (-a_2(-t),-a_1(-t),$ $b(-t),c(-t),-f(-t))$ is another symmetry
of (\ref{HolRedN11}). Since it maps a solution with $b'(0)=1$
into a solution with $b'(0)=-1$, we only need to consider the case
$b'(0)=1$ and obtain

\begin{equation}
\label{PowSerN11S3C}
\begin{array}{l}
\begin{array}{rclrrrrrrrl}
a_1(t) & = & a_0 & + & 0\cdot t & + & \tfrac{1}{a_0}t^2 & + &
0\cdot t^3 & + & \tfrac{3q-23}{24a_0^3} t^4 \\
\end{array} \\ \\
\begin{array}{rrrl}
\quad\quad\quad\quad\quad\quad\!\! & + & 0\cdot t^5 & + \ldots
\end{array}\\ \\
\begin{array}{rclrrrrrrrl}
a_2(t) & = & -a_0 & + & 0\cdot t & + & \tfrac{q-2}{a_0}t^2 & + &
0\cdot t^3 & - & \tfrac{12q^2 - 25q - 7}{24a_0^3} t^4 \\
\end{array} \\ \\
\begin{array}{rrrl}
\quad\quad\quad\quad\quad\quad\quad\!\!\! & + & 0\cdot t^5 & +
\ldots
\end{array}\\ \\
\begin{array}{rclrrrrrrrl}
b(t)\:\: & = & 0 & + & t & + & 0\cdot t^2 & - & \tfrac{1}{6a_0^2}
t^3 & + & 0\cdot t^4 \\
\end{array} \\ \\
\begin{array}{rrrl}
\quad\quad\quad\quad\quad
& - & \tfrac{39q^2 - 114q +25}{240a_0^4} t^5 & + \ldots \\
\end{array} \\ \\
\begin{array}{rclrrrrrrrrrrl}
c(t)\:\: & = & a_0 & + & 0\cdot t & + & \tfrac{q}{2a_0}t^2 & + &
0\cdot t^3 & + & \tfrac{3q^2 - 13q +3}{24a_0^3}t^4 \\
\end{array} \\ \\
\begin{array}{rrrl}
\quad\quad\quad\quad\quad\quad\!\! & + & 0\cdot t^5 & + \ldots
\end{array}\\ \\
\begin{array}{rclrrrrrrrl}
f(t)\:\: & = & 0 & + & 6t & + & 0\cdot t^2 & - &
\tfrac{4}{a_0^2}t^3 & + & 0\cdot t^4 \\
\end{array} \\ \\
\begin{array}{rrrl}
\quad\quad\quad\quad\quad\: & + & \tfrac{3q^2 - 18q +175}{20a_0^4}
t^5 & + \ldots \\
\end{array}
\end{array}
\end{equation}

As usual, $q$ can be chosen freely. Our aim is to prove that the three initial value
problems have a unique smooth solution for any choice of the parameters $a_0$, $b_0$,
$q$, $q_1$, and $q_2$. We therefore have to check the smoothness conditions for
the above power series. In \cite{Rei}, we have proven that an analytic diagonal metric $g=g_t
+ dt^2$ of cohomogeneity one has a smooth extension to a singular orbit at $t=0$ if

\begin{enumerate}
    \item $g_t$ converges for $t\rightarrow 0$ to a degenerate bilinear form which
    is invariant with respect to the cohomogeneity-one action.
    \item The sectional curvature of the collapsing sphere behaves as $\tfrac{1}{t}
    + O(1)$ for $t\rightarrow 0$.
    \item The coefficient functions of the horizontal part are even.
    \item The coefficient functions of the vertical part are odd.
\end{enumerate}

We remark that this result can also be applied to orbifold metrics. The metric
on the singular orbit $\mathbb{CP}^2$ is in all three cases the Fubini study
metric and thus $SU(3)$-invariant.

In order to check the second smoothness condition, which we have
also mentioned in Remark \ref{SmoothnessRemark}, we have to search
for the metric $h$ with constant sectional curvature $1$ on
$K/U(1)_{k,l}$. Let $\widetilde{h}$ with
$\widetilde{h}(X,Y)=-\tfrac{1}{2} \text{tr}(XY)$ for all
$X,Y\in\mathfrak{su}(2)$ be the metric with sectional curvature
$1$ on $SU(2)$. We embed $SU(2)$ into $SU(3)$ such that its Lie
algebra becomes $\mathfrak{u}(1)_{1,-1}\oplus V_1$. The map
$\pi:SU(2)\rightarrow K/U(1)_{k,l}$ with $\pi(k):=k U(1)_{k,l}$ is
a covering map. $h$ therefore has to satisfy

\begin{equation}
\|d\pi(X)\|_h = \|X\|_{\widetilde{h}}\:.
\end{equation}

With the help of this formula, we see that

\begin{enumerate}
    \item in the case where $a(0)=f(0)=0$, $\|e_1\|_q=\|e_2\|_q=1$ and
    $\|e_7\|_q=|\tfrac{2\Delta}{k+l}|$. We therefore obtain $|a'(0)|=1$ and
    $|f'(0)|=|\tfrac{2\Delta}{k+l}|$.
    \item in the case where $a_1(0)=a_2(0)=f(0)=0$, $\|e_1\|_q=\|e_2\|_q=1$ and
    $\|e_7\|_q=3$. We therefore obtain $|a_1'(0)|=|a_2'(0)|=1$ and
    $|f'(0)|=3$.
    \item in the case where $b(0)=f(0)=0$, $\|e_3\|_q=\|e_4\|_q=1$ and
    $\|e_7\|_q=6$. We therefore obtain $|b'(0)|=1$ and $|f'(0)|=6$.
\end{enumerate}

Here, $q$ denotes the biinvariant metric on $\mathfrak{su}(3)$
with $q(X,Y)=-\tfrac{1}{2} \text{tr}(XY)$  which we have
introduced earlier.

The power series (\ref{PowSerS3Zkl}), (\ref{PowSerN11S3}),
(\ref{PowSerN11S3B}), and (\ref{PowSerN11S3C}) obviously satisfy
the first two smoothness conditions. As in the previous two
sections, we can prove that the initial value problems have a
unique power series solution for any choice of $a_0$, $b_0$, $q$,
$q_1$, and $q_2$. We are now able to prove the remaining
smoothness conditions by means of symmetry arguments.

\begin{enumerate}
    \item $(a(t),b(t),c(t),f(t))\mapsto (-a(-t),b(-t),c(-t),-f(-t))$ is
    a symmetry of (\ref{HolRedNkl}). Any solution of that system
    with $a(0)=f(0)=0$, $b(0)=c(0)=b_0$, and
    $f'''(0)=\tfrac{q}{b_0^2}$ is mapped by the symmetry to
    another solution with the same initial values. $a$ and $f$
    therefore are odd functions and $b$ and $c$ are even. The power series
    (\ref{PowSerS3Zkl}) thus satisfies all smoothness conditions.
    \item The smoothness of (\ref{PowSerN11S3}) can be proven
    with the help of the symmetry $(a_1(t),a_2(t),b(t),c(t),f(t)) \mapsto
    (-a_1(-t),-a_2(-t),b(-t),c(-t),$ $-f(-t))$ of (\ref{HolRedN11}).
    The relation $b(t)=c(t)$ follows with the help of the symmetry
    $(b,c)\mapsto(c,b)$.
    \item The smoothness of (\ref{PowSerN11S3B}) can be proven
    with the help of the symmetry $(a_1(t),a_2(t),b(t),c(t),f(t)) \mapsto
    (a_1(-t),a_2(-t),-b(-t),c(-t),$ $-f(-t))$ of (\ref{HolRedN11}) and
    the relation $a_1(t)=a_2(t)$ follows with the help of the symmetry
    $(a_1,a_2)\mapsto(a_2,a_1)$.
    \item The smoothness of (\ref{PowSerN11S3C}) can be proven
    with the help of the symmetry $(a_1(t),a_2(t),b(t),c(t),f(t)) \mapsto
    (a_2(-t),-a_1(-t),-b(-t),c(-t),$ $-f(-t))$ of (\ref{HolRedN11}).
\end{enumerate}

We finally have to prove that the power series converge. In the
setting of this section, Assumption \ref{WEschAssumpt} is not
always satisfied. If $k=l=1$ and $b(0)=f(0)=0$, $\mathfrak{p}$ and
$\mathfrak{p}^{\perp}$ both contain a trivial
$U(1)_{1,1}$-submodule. The spaces $\mathfrak{p}$ and
$\mathfrak{p}^{\perp}$ also contain a common submodule if $k=1$,
$l=0$, and $a(0)=f(0)=0$. We can nevertheless prove the
convergence with the help of the arguments which we have made in
Remark \ref{WEschThmRem}.\ref{ConvRem}.

Since \ref{WEschAssumpt} is in some cases not satisfied and we
have not described the spaces $W_m^h$ and $W_2^v$ explicitly, we
will not study the existence of Einstein metrics on a tubular
simultaneous of $\mathbb{CP}^2$. We conclude this section by
summarizing our results on metrics with special holonomy.

\begin{Th}
Let $M$ be a cohomogeneity-one orbifold whose principal orbit is $N^{k,l}$.
We assume that $M$ has exactly one singular orbit of type $\mathbb{CP}^2$,
which is at $t=0$.

\begin{enumerate}
    \item Let $N^{k,l}$ be not $SU(3)$-equivariantly diffeomorphic to
    $N^{1,1}$ and let $k+l\neq 0$. For any $b_0\in\mathbb{R}
    \setminus\{0\}$ and $q\in\mathbb{R}$ there exists a unique
    $SU(3)$-invariant parallel Spin($7$)-structure $\Omega$ on a
    tubular simultaneous of $\mathbb{CP}^2$ which is determined
    by a basis of type (\ref{CBasisNkl}) and satisfies $a(0)=f(0)=0$,
    $b(0)=c(0)=b_0$, and $f'''(0)=\tfrac{q}{b_0^2}$. The holonomy
    of the associated metric is all of Spin($7$).

    \item Let $k=l=1$. For any $b_0\in\mathbb{R}\setminus\{0\}$
    and $q_1,q_2\in\mathbb{R}$ there exists a unique $SU(3)$-invariant
    parallel Spin($7$)-structure $\Omega$ on a tubular simultaneous of
    $\mathbb{CP}^2$ which is determined by a basis of type
    (\ref{CBasisN11}) and satisfies $a_1(0)=a_2(0)=f(0)=0$,
    $b(0)=c(0)=b_0$, $a_1'''(0)=\tfrac{q_1}{b_0^2}$ and
    $a_2'''(0)=\tfrac{q_2}{b_0^2}$. Moreover, we have $b(t) = c(t)$
    for all values of $t$.

    \item Let $k=l=1$. For any $a_0\in\mathbb{R}\setminus\{0\}$
    and $q\in\mathbb{R}$ there exists a unique $SU(3)$-invariant
    parallel Spin($7$)-structure $\Omega$ on a tubular simultaneous of
    $\mathbb{CP}^2$ which is determined by a basis of type
    (\ref{CBasisN11}) and satisfies $b(0)=f(0)=0$,
    $a_1(0)=a_2(0)=c(0)=a_0$, and $b'''(0)=\tfrac{q}{a_0^2}$.
    Moreover, we have $a_1(t)=a_2(t)$ for all values of $t$.

    \item Let $k=l=1$. For any $a_0\in\mathbb{R}\setminus\{0\}$
    and $q\in\mathbb{R}$ there exists a unique $SU(3)$-invariant
    parallel Spin($7$)-structure $\Omega$ on a tubular simultaneous of
    $\mathbb{CP}^2$ which is determined by a basis of type
    (\ref{CBasisN11}) and satisfies $b(0)=f(0)=0$,
    $a_1(0)=-a_2(0)=c(0)=a_0$, $b'(0)=1$, and $c''(0)=\tfrac{q}{a_0}$.
\end{enumerate}
\end{Th}

\begin{Rem}
\begin{enumerate}
    \item The first class of metrics from the above theorem was considered
    by Cveti\v{c} \cite{Cve} and by Kanno, Yasui \cite{KanI}. The second
    class of metrics was discovered independently of the
    author by Bazaikin \cite{Baz1}. We have interpreted the
    parameters on which these metrics depend in terms of
    the two initial conditions $q_1$ and $q_2$ of
    third order. Moreover, we have proven that no further metrics
    of this kind with $b(t)\neq c(t)$ exist. In the second
    paper of Kanno, Yasui \cite{KanII}, a power series ansatz for the
    third and the fourth class of metrics was made. However,
    our proofs of the smoothness and the convergence of the power series
    are new.
    \item In the cases where $N^{k,l}$ is generic or $k=l=1$ and $a_1(0)
    =a_2(0)=f(0)=0$ Assumption \ref{WEschAssumpt} is satisfied. By calculating
    $W_2^v/W_0^v$ we see that the free parameters $q$, $q_1$, $q_2$ are
    indeed a subset of the free parameters from Theorem
    \ref{EinsteinCohom1Solution}. Although we do not know if
    Assumption \ref{WEschAssumpt} is necessary for Theorem
    \ref{EinsteinCohom1Solution} to be true, we can make a similar
    observation in the other two cases from the above theorem. All in
    all, we hope to have shed some light on the origin of these
    parameters.
\end{enumerate}
\end{Rem}

\end{document}